\documentclass[12pt]{amsart}  

\usepackage[latin1]{inputenc}
\usepackage{amsmath} 
\usepackage{amsfonts}
\usepackage{amssymb}
\usepackage{stmaryrd}
\usepackage{latexsym} 
\usepackage{graphicx}
\usepackage{subfigure}
\usepackage{color}
\usepackage{hyperref}
\usepackage{verbatim}
\usepackage[all]{xy}
\usepackage{graphics}
\usepackage{pdfsync}

\theoremstyle{definition}
\newtheorem{theorem}{Theorem}
\newtheorem{proposition}[theorem]{Proposition}

\newtheorem{lemma}[theorem]{Lemma}
\newtheorem{corollary}[theorem]{Corollary}
\newtheorem{conjecture}[theorem]{Conjecture}
\newtheorem{definition}[theorem]{Definition}
\newtheorem{example}[theorem]{Example}

\theoremstyle{remark}
\newtheorem{remark}[theorem]{Remark}
\newtheorem{notation}[theorem]{Note}

\newcommand{\ra}{r_\alpha}
\newcommand{\rb}{r_\beta}
\newcommand{\bd}{P_{R,S}}

\newlength\cellsize 
\savebox2{%
\begin{picture}(15,15)
\put(0,0){\line(1,0){15}}
\put(0,0){\line(0,1){15}}
\put(15,0){\line(0,1){15}}
\put(0,15){\line(1,0){15}}
\end{picture}}
\newcommand\cellify[1]{\def\thearg{#1}\def\nothing{}%
\ifx\thearg\nothing
\vrule width0pt height\cellsize depth0pt\else
\hbox to 0pt{\usebox2\hss}\fi%
\vbox to 15\unitlength{
\vss
\hbox to 15\unitlength{\hss$#1$\hss}
\vss}}
\newcommand\tableau[1]{\vtop{\let\\=\cr
\setlength\baselineskip{-16000pt}
\setlength\lineskiplimit{16000pt}
\setlength\lineskip{0pt}
\halign{&\cellify{##}\cr#1\crcr}}}
\savebox3{%
\begin{picture}(15,15)
\put(0,0){\line(1,0){15}}
\put(0,0){\line(0,1){15}}
\put(15,0){\line(0,1){15}}
\put(0,15){\line(1,0){15}}
\end{picture}}
\newcommand\expath[1]{%
\hbox to 0pt{\usebox3\hss}%
\vbox to 15\unitlength{
\vss
\hbox to 15\unitlength{\hss$#1$\hss}
\vss}}
\newcommand\bas[1]{\omit \vbox to \cellsize{ \vss \hbox to \cellsize{\hss$#1$\hss} \vss}}

\begin{document}

\title[Necessary conditions for Schur-maximality]{Necessary conditions for Schur-maximality}

\author{Foster Tom}
\address{
 Department of Mathematics,
 University of British Columbia,
 Vancouver BC V6T 1Z2, Canada}
\email{foster@math.ubc.ca}

\author{Stephanie van Willigenburg}
\address{
 Department of Mathematics,
 University of British Columbia,
 Vancouver BC V6T 1Z2, Canada}
\email{steph@math.ubc.ca}

\thanks{
Both  authors were supported in part by the National Sciences and Engineering Research Council of Canada.}

\begin{abstract} McNamara and Pylyavskyy conjectured precisely which connected skew shapes are maximal in the Schur-positivity order, which says that $B\leq _s A$ if $s_A-s_B$ is Schur-positive. Towards this, McNamara and van Willigenburg proved that it suffices to study equitable ribbons, namely ribbons whose row lengths are all of length $a$ or $(a+1)$ for $a\geq 2$. In this paper we confirm the  conjecture of McNamara and Pylyavskyy in all cases where the comparable equitable ribbons form a chain. We also confirm a conjecture of McNamara and van Willigenburg regarding which equitable ribbons in general are minimal. 

Additionally, we establish two sufficient conditions for the difference of two ribbons to be Schur-positive, which manifest as diagrammatic operations on ribbons. We also deduce two necessary conditions for the difference of two equitable ribbons to be Schur-positive that rely on rows of length $a$ being at the end, or on rows of length $(a+1)$ being evenly distributed.
\end{abstract}

\maketitle
\tableofcontents
\section{Introduction}\label{sec:intro}
Within the algebra of symmetric functions, perhaps the most acclaimed basis is that consisting of Schur functions due to their ubiquitous nature: arising in enumerative combinatorics as generating functions for tableaux, in the representation theory of the symmetric and general linear groups, and in algebraic geometry when studying the cohomology ring of the Grassmannian, in addition to other areas such as quantum physics. One vibrant research avenue concerning them is that of determining when a symmetric function is Schur-positive, that is, when a symmetric function expanded as a linear combination of Schur functions has nonnegative coefficients. Schur-positive functions have two particular representation-theoretic interpretations. The first is that if a homogeneous symmetric function $f$ of degree $N$ is Schur-positive, then it arises as the Frobenius image of some representation of the symmetric group $S_N$. The second is that $f(x_1, \ldots, x_n)$ is the character of a polynomial representation of the general linear group $\mathit{GL}(n, \mathbb{C})$.  Consequently, knowing when a symmetric function is Schur-positive is  very desirable. However, this phenomenon is incredibly rare, as has been calculated by F. Bergeron, Patrias and Reiner \cite{FBergeron} who proved that the probability of a  monomial-positive symmetric function, which is homogeneous of degree $N$, to be Schur-positive is
$$\prod _{\lambda} \left( \sum _{\mu} K_{\lambda\mu} \right)^{-1}$$
where $\lambda$ and $ \mu$ are partitions of $N$ and $K_{\lambda\mu}$ is the Kostka coefficient. To help explain this phenomenon, in addition to dealing with Schur functions directly, tools have been developed to determine Schur-positivity such as dual equivalence graphs \cite{Assaf, Assaf2, BlasiakFomin}, or the theory of crystal bases that has also been applied as a means to determine Schur-positivity \cite{SchillingBook}.

Despite their rarity, examples of Schur-positive functions arise in a variety of contexts from graph theory and the study of chromatic symmetric functions \cite{Gasharov} and chromatic quasisymmetric functions \cite{ShareshianWachs} to enumerative combinatorics where sets of partitions determining a quasisymmetric function have been shown to determine a function that is in fact symmetric and Schur-positive \cite{Elizalde, Ges84, GesR}. However, the most classic rendition of the question of Schur-positivity is to determine when the difference  of two skew Schur functions is Schur-positive. This question is still considered to be intractable in full generality, however partial results and conjectures do exist \cite{BBR06, FFLP05, KWvW08,  Kir04, lpp, LLT97, McN08, McvW09b, Oko97} and more progress has been made when the difference is zero \cite{btvw, Gut09, McvW09a, RSvW07,  vWi05}. One particularly intriguing conjecture is due to McNamara and Pylyavskyy, who gave a construction for connected skew shapes $A_i$ such that for any other connected skew shape $B$ the difference of skew Schur functions $s_{A_i}-s_B$ is Schur-positive for some $A_i$ with the same number of cells as $B$. These $A_i$ were ribbons, that is, connected skew shapes with no $2\times2$ subdiagram, and in \cite{maxsupport} McNamara and van Willigenburg showed that such a connected skew shape must indeed be a ribbon whose row lengths are all $a$ or $(a+1)$. We extend these results by establishing additional necessary conditions on such ribbons and by proving the conjecture of McNamara and Pylyavskyy in the case where the set of ribbons whose row lengths are a fixed number of $a$'s and $(a+1)$'s is totally ordered by $A\geq_s B$ if $s_A-s_B$ is Schur-positive. That is, while not confirming the conjecture in full generality we do confirm it in the important case of chains. 

This article is structured as follows. In the next section we introduce necessary definitions and concepts, and prove that the ribbon conjectured to be maximal is indeed an equitable ribbon in Lemma~\ref{lem:eqrib} and Corollary~\ref{cor:eqrib}. In Section~\ref{sec:tools} we introduce two operations on ribbons that guarantee that the difference of two ribbon Schur functions is Schur-positive. More precisely, in Theorems~\ref{thm:moving1} and \ref{thm:moving1cor} we create a larger ribbon using an operation that moves cells from the top row of a ribbon. Meanwhile in Theorem~\ref{thm:celltr} we create a larger ribbon using an operation that exchanges two rows of a ribbon. We then apply these tools in Section~\ref{sec:chains} when we identify all chains of equitable ribbons in Theorem~\ref{thm:chains}, thereby confirming McNamara and Pylyavskyy's conjecture in these cases. In Section~\ref{sec:shortends} we confirm another case of this conjecture in Corollary~\ref{cor:shortends}, which is reliant on Theorem~\ref{thm:shortends} that a larger equitable ribbon must have short rows as its first and last rows. In Section~\ref{sec:coarsenings} we introduce the notions of profile and quasi-profile, the latter of which is particularly crucial as it yields another condition  for the Schur-positivity of a difference of equitable ribbons in Theorem~\ref{thm:smalls}, which implies that a larger equitable ribbon must have its longer rows being more evenly distributed,  confirms that the chains identified in 
 Theorem~\ref{thm:chains} are the only ones in Corollary~\ref{cor:incompar}, and confirms the conjecture of McNamara and van Willigenburg on minimal equitable ribbons in Theorem~\ref{thm:minimal}.

\section{Background}\label{sec:background}
\subsection{Compositions and partitions}\label{subsec:compositions} We say that a sequence of positive integers $\alpha = \alpha _1 \cdots \alpha _\ell$ is a \emph{composition} and call the $\alpha _i$ its \emph{parts}. If $\sum _{i=1} ^{\ell} \alpha _i = N$ then we say that $N$ is the \emph{size} of $\alpha$, denoted by $|\alpha |$, and say that $\ell$ is the \emph{length} of $\alpha$, denoted by $\ell (\alpha)$. If $\alpha _{i+1} = \alpha _{i+2} = \cdots = \alpha _{i+m} = j$ then we often abbreviate this to $j^m$. Given two compositions $\alpha = \alpha _1 \cdots \alpha _{\ell(\alpha)}$ and $\beta = \beta _1 \cdots \beta _{\ell(\beta)}$ we define the \emph{reversal} of $\alpha$ to be
$$\alpha ^\ast = \alpha _{\ell(\alpha)} \cdots \alpha _1$$the \emph{concatenation} of $\alpha$ and $\beta$ to be
$$\alpha\cdot \beta = \alpha _1 \cdots \alpha _{\ell(\alpha)}\beta _1 \cdots \beta _{\ell(\beta)}$$the \emph{near-concatenation} of $\alpha$ and $\beta$ to be
$$\alpha\odot \beta = \alpha _1 \cdots (\alpha _{\ell(\alpha)}+\beta _1) \cdots \beta _{\ell(\beta)}$$and the \emph{composition of compositions} \cite[Section 3.1]{btvw} $\alpha$ and $\beta$ to be
$$\alpha \circ \beta = \beta ^{\odot \alpha _1}\cdot \beta ^{\odot \alpha _2} \cdots \beta ^{\odot \alpha _{\ell(\alpha)}}$$where $\beta ^{\odot \alpha _i}$ denotes the near-concatenation of $\alpha _i$ copies of $\beta$. With this in mind it is straightforward to prove  that
\begin{equation}\label{eq:circsize}
|\alpha \circ \beta| = |\alpha||\beta|.
\end{equation}

\begin{example}\label{ex:compofcomp}
Let $\alpha = 12$ and $\beta = 31$. Then $\alpha ^\ast = 21$, $\alpha \cdot \beta = 1231$, $\alpha \odot \beta = 151$ and
$$\alpha \circ \beta = 31 \cdot (31\odot 31) = 31341.$$
\end{example}

Any composition whose parts when read from left to right are weakly decreasing is called a \emph{partition}. In our above example $\beta = 31$ is a partition. Given partitions $\lambda$ and $\mu$ such that $\ell = \max \{ \ell(\lambda), \ell(\mu) \}$, append zeroes to $\lambda$ or $\mu$ to artificially increase the length so that $\ell(\lambda)= \ell(\mu)$. If $\lambda _i = \mu _i$ for all $1\leq i < j$ and $\lambda _j \neq \mu _j$, then if $\lambda _j >\mu _j$ we say that $\lambda$ is \emph{lexicographically greater} than $\mu$, denoted by $\lambda >_{lex} \mu$. Also note that every composition $\alpha$ \emph{determines} a partition $\lambda(\alpha)$ by reordering the parts of $\alpha$ into weakly decreasing order. Lastly, for convenience, we denote by $\emptyset$ the empty composition or partition of size and length 0.

\subsection{Skew shapes and ribbons}\label{subsec:ribbons}
Given a partition $\lambda$, its \emph{diagram}, also denoted by $\lambda$, is the array of left-justified cells with $\lambda _i$ cells in row $i$ from the \emph{top}. If we refer to cell $(i,j)$ then this refers to the cell in the $i$-th row from the top and $j$-th column from the left. The \emph{only} exception to this is Subsection~\ref{subsec:eqrib} where for ease of exposition this will refer to the cell in the $i$-th row from the bottom and $j$-th column from the left.
Let $\lambda$ and $\mu$ be partitions such that $\ell (\lambda )\geq \ell (\mu)$ and $\lambda _i \geq \mu _i$ for all $1\leq i \leq \ell(\mu)$. Then we define the \emph{skew shape} $\lambda /\mu$ to be the array of cells
$$\lambda / \mu = \{ (i,j) : (i,j) \in \lambda, (i,j) \not\in \mu \}.$$We call $\lambda$ the \emph{outer shape} and $\mu$ the \emph{inner shape}. Also by \emph{row} (respectively, \emph{column}) \emph{length} we refer to the number of cells in a given row (respectively, column) of $\lambda /\mu$. We define the \emph{transpose} of $\lambda / \mu$   to be the array of cells
$$(\lambda / \mu )^t = \{ (j,i) : (i,j) \in \lambda, (i,j) \not\in \mu \}.$$
Skew shapes can be either disconnected or connected, where a skew shape is said to be \emph{disconnected} if it can be partitioned into nonempty skew shapes $D_1$ and $D_2$ so no cell in $D_1$ has a row or column in common with any cell in $D_2$. Otherwise $\lambda / \mu$ is said to be \emph{connected}. Connected skew shapes that will be our focus later are \emph{ribbons}, which are connected skew shapes not containing the subdiagram $22 = \setlength{\unitlength}{2mm}
\begin{picture}(2,2)(0,0.4)
\multiput(0,0)(0,1){3}{\line(1,0){2}}
\multiput(0,0)(1,0){3}{\line(0,1){2}}
\end{picture}\ $. Observe that, therefore, a ribbon is completely determined by its row lengths read from top to bottom. This yields a composition, and henceforth we abuse notation by identifying a ribbon with its corresponding composition.

\begin{example}\label{ex:ribbon}
Below is the skew shape $\lambda/\mu = 4431/32$. Note that it is also the ribbon $1231$.
$$\tableau{&&&\ \\&&\ &\ \\\ &\ &\ \\\ }$$
\end{example}

Given a ribbon, we refer to the top row and the bottom row as the \emph{end rows} and every other row as an \emph{intermediate row}. Also note that given a ribbon with $R$ rows, $S$ columns and $N$ cells in total that
$$R+S=N+1$$and knowing any two of these statistics will determine the third.

Narrowing our focus even further, we now turn our attention to equitable ribbons. We describe a ribbon as \emph{row-equitable} if all its row lengths are $a$ or $(a+1)$ for some $a \geq 1$ and \emph{column-equitable} if all its column lengths are $b$ or $(b+1)$ for some $b\geq 1$. A ribbon is said to be \emph{equitable} if it is both row-equitable and column-equitable. Given such an equitable ribbon we call the rows of length $a$ \emph{short rows} and of length $(a+1)$ \emph{long rows}. Furthermore, given an equitable ribbon $\alpha$ we denote by $SE(\alpha)$ the number of end rows of $\alpha$ that are short. We also note that if a ribbon is row-equitable and $a\geq 2$, then we are guaranteed that it is equitable, and similarly if we are given that a ribbon is column-equitable and $b\geq 2$, then again we are guaranteed that it is equitable.

\begin{example}\label{ex:eqribbon} The ribbon in Example~\ref{ex:ribbon} is column-equitable but not row-equitable. The ribbons $\alpha = 232$ and $\beta = 233$ below are both equitable with $SE(\alpha) = 2$ and $SE(\beta) = 1$, respectively.
$$\alpha = \tableau{&&&\ &\ \\ & \ &\ &\ \\\ &\ } \quad\quad \beta = \tableau{&&&&\ &\ \\&&\ &\ &\ \\\ &\ &\ }$$
\end{example}

\subsection{Skew Schur functions and ribbon Schur functions}\label{subsec:ssfandrsf}
Now that we have introduced skew shapes and ribbons, we can use them to define our algebraic focus, namely skew Schur functions. But before this we need to introduce tableaux.

Let $\lambda/\mu$ be a skew shape. Then a \emph{semistandard Young tableaux (SSYT)} $T$ of \emph{shape} $\lambda/\mu$, denoted by $sh(T)=\lambda / \mu$, is a filling of the cells of $\lambda/\mu$ with positive integers such that the entries in each row weakly increase when read from left to right, and the entries in each column strictly increase when read from top to bottom. The entry of cell $(i,j)$ in an SSYT $T$ is denoted by $T_{i,j}$ and the \emph{content} of $T$ is
$$c(T)=c_1(T)c_2(T)\cdots$$where $c_i(T)$ is the number of $i$'s appearing in $T$.

\begin{example}\label{ex:tableau}If $T=\tableau{&&&1&1&\\&1&2&2\\1&3}$ then $sh(T)=542/31$, $c(T)=421$, and $ T_{3,2} = 3$.
\end{example}

The \emph{skew Schur function} $s_{\lambda / \mu}$ in variables $\{x_1, x_2, \ldots\}$ is then defined to be 
$$s_{\lambda / \mu} = \sum _T x^{c(T)}$$where the sum is over all SSYTs $T$ with $sh(T) = \lambda / \mu$ and
$$x^{c(T)}=x_1 ^{c_1(T)}x_2^{c_2(T)}\cdots .$$For example, $x^{c(T)}= x_1^4x_2^2x_3$ for the above $T$.  Two types of skew Schur functions will be of particular interest to us. The first is when $\lambda / \mu$ is a ribbon corresponding to some composition $\alpha$, then we define the \emph{ribbon Schur function} $r_\alpha$ to be the skew Schur function $s_{\lambda / \mu}$. The second of these is when $\lambda / \mu$ has $\mu = \emptyset$, then we define the \emph{Schur function} $s_\lambda$ to be the skew Schur function $s_{\lambda / \mu}$. Schur functions are of particular interest since they form a basis for the algebra of symmetric functions to which skew Schur functions belong, and hence every skew Schur function can be written as a linear combination of Schur functions. To see exactly how we require Littlewood-Richardson tableaux. We say that a tableau is a \emph{Littlewood-Richardson (LR) tableau} if, as we read the entries from right to left along each row, taking the rows from top to bottom, the number of $i$'s we have read is always weakly greater than the number of $(i+1)$'s we have read, for all $i\geq 1$. This is known as the \emph{lattice word condition}. See Example~\ref{ex:tableau} for an example of an LR tableau. Note that the lattice word condition guarantees that every cell in the top row of an LR tableau is filled with a 1. This seemingly innocuous consequence will play a powerful role later.

\begin{theorem}[Littlewood-Richardson rule]\label{thm:LRrule}  \cite{LiRi34, Sch77, Tho78}  Let $\lambda/\mu$ be a skew shape. Then 
$$s_{\lambda / \mu} = \sum _\nu c ^\lambda _{\mu\nu} s_\nu$$where $c ^\lambda _{\mu\nu}$ is the number of LR tableaux $T$ satisfying $sh(T)= \lambda / \mu$ and $c(T) = \nu$.
\end{theorem}

It is well-known that the LR tableau obtained by filling the $i$-th cell in each column with $i$ yields the lexicographically largest $\nu$ such that $ c^\lambda _{\mu\nu} \neq 0$, and moreover $ c^\lambda _{\mu\nu} = 1$. The Littlewood-Richardson rule also guarantees us that skew Schur functions are \emph{Schur-positive}, that is, when expanded as a linear combination of Schur functions all the coefficients are nonnegative, and this leads us to our key definition.

\begin{definition}\label{def:spositive} Given symmetric functions $F$ and $G$ we say that
$$F\geq _s G$$if $F-G$ is Schur-positive.
\end{definition}
We will now use this to define a poset on skew shapes. Given a skew shape $\lambda/\mu$ let $[\lambda / \mu]$ be the equivalence class consisting of all skew shapes whose skew Schur function is equal to $s_{\lambda / \mu}$. Then say 
$$[\lambda / \mu] \geq _s [\nu / \rho]$$if $s_{\lambda / \mu} \geq _s s_{\nu / \rho}$. Since, by definition, $s_{\lambda / \mu}$ is homogeneous of degree $N$, where $N$ is the number of cells in $\lambda/ \mu$ it follows that $[\lambda / \mu]$ and  $[\nu / \rho]$ will be incomparable unless $\lambda / \mu$ and $\nu / \rho$ have the same number of cells, $N$. Hence now consider all the equivalence classes $[\lambda / \mu]$ where $\lambda/\mu$ has $N$ cells, and let $\mathcal{P}_N$ be the poset whose elements are these equivalence classes, with order relation $\leq _s$.

In general little is known about the equivalence class $[\lambda / \mu]$ though special cases have been considered, for example  \cite{Gut09, vWi05}, and in particular the elements of $[\lambda / \mu]$ have been completely determined when  $\lambda / \mu$ is a ribbon \cite[Theorem 4.1]{btvw}.

\begin{theorem}\label{thm:hdl} \cite[Theorem 4.1]{btvw} Two compositions $\alpha$ and $\beta$ satisfy $r_\alpha = r_\beta$ if and only if for some $k$ we can decompose
$$\alpha = \alpha ^1 \circ \cdots \circ \alpha ^k \mbox{ and } \beta = \beta ^1 \circ \cdots \circ \beta ^k$$where for each $i$, $1\leq i\leq k$, either $\beta ^i = \alpha ^i$ or $\beta ^i = (\alpha ^i)^\ast$. In particular, $r_\alpha = r_ {\alpha ^\ast}$.
\end{theorem}

As noted, $r_\alpha = r_ {\alpha ^\ast}$, and furthermore if $\alpha$ is an equitable ribbon then the following corollary confirms that this is all.

\begin{corollary} \label{cor:ineq}Let $\alpha$ be an equitable ribbon. Then the equivalence class of $\alpha$ only contains $\alpha$ and $\alpha ^\ast$.\end{corollary}
 
 \begin{proof} We show that $\alpha=\alpha_1\cdots\alpha_{\ell(\alpha)}$ does not factorize as a composition $\beta\circ \gamma$ of two nonsymmetric compositions $\beta$ and $\gamma$. Suppose that it does. We know that the row lengths of $\alpha$ are all either $a$ or $(a+1)$ for some $a$. Now because $\alpha$ is a concatenation of near-concatenations of $\gamma$, we must have $\gamma_1=\alpha_1\in\{a,a+1\}$ and $\gamma_{\ell(\gamma)}=\alpha_{\ell(\alpha)}\in\{a,a+1\}$. Also, since $\beta$ is nonsymmetric and hence not all $1$'s there will be some near-concatenation, so $\gamma_1+\gamma_{\ell(\gamma)}\in\{a,a+1\}$; the only possibility is if $$\gamma_1=\alpha_1=\gamma_{\ell(\gamma)}=\alpha_{\ell(\alpha)}=a=1,\hbox{ so that }\gamma_1+\gamma_{\ell(\gamma)}=a+1=2.$$ In particular, the first row of $\alpha$ has length $1$. Now by applying the $\omega$ involution $\omega (s_{\lambda / \mu}) = s_{(\lambda / \mu)^t}$, we can use the same argument to show that the first column of $\alpha$ has length $1$. However, now $\alpha$ is just a single cell, so can not factorize as described.\end{proof}
 
 Returning to $\mathcal{P}_N$ it is straightforward to verify certain properties of it, such as by the Littlewod-Richardson rule $\mathcal{P}_N$ has a unique maximal element consisting of $N$ disconnected components each containing a single cell. Therefore, the more interesting question is what is a maximal element among connected skew shapes? If we restrict our attention to connected skew shapes then it was shown \cite[Corollary 3.6]{maxsupport} that the skew shape must be an equitable ribbon, but the question still remains, which one? It is this question that will serve as the impetus of our paper, and hence before we continue we will make two reductions in order to simplify our study.
 
 \textbf{Reduction 1:} This reduction is notational and we will henceforth identify the equivalence class $[\alpha]$ with its representative $\alpha$. This is because by Corollary~\ref{cor:ineq} the equivalence class of an equitable ribbon $\alpha$ only contains $\alpha$ and $\alpha ^\ast$.
 
 \textbf{Reduction 2:} This reduction fixes the number of rows of any ribbon we consider to some integer $R$. This is because by \cite[Lemma 3.8]{maxsupport} two ribbons, and hence two equitable ribbons, are incomparable in $\mathcal{P}_N$ if they have a different number of rows.
 
Hence we will focus our attention on the convex subposet of $\mathcal{P}_N$, which we will denote by $\mathcal{R}((a+1)^na^m)$, which consists of all equitable ribbons with $n$ rows of length $(a+1)$ and $m$ rows of length $a$. By the involution $\omega (s_{\lambda / \mu}) = s_{(\lambda / \mu)^t}$ it additionally suffices to consider only $a\geq 2$. However, many of our results intriguingly hold when $a=1$, and hence we include these results for completeness. With these reductions in place we now define the box diagonal diagram, which is crucial to the question of maximal connected skew shapes in $\mathcal{P}_N$.

\begin{definition}\label{def:bd} The \emph{box diagonal diagram} $\bd$ is the skew shape consisting of $R$ rows, $S$ columns and $N = R+S-1$ cells constructed as follows. Draw a grid $R$ cells high and $S$ cells wide, and a line $L$ given by $y= \frac{R}{S} x$ from the bottom left corner to the top right corner. Then $\bd$ consists of the cells whose interior or top left corner point is intercepted by $L$.
\end{definition}

\begin{example}\label{ex:bd} The box diagonal diagrams $P_{5,7}$ and $P_{3,6}$ are indicated below by the shaded blue and green cells. Note that $P_{3,6}$ illustrates the ``top left corner point'' phrase of the definition indicated by the green cells.
\begin{figure}[h]
\includegraphics[width=0.6\textwidth]{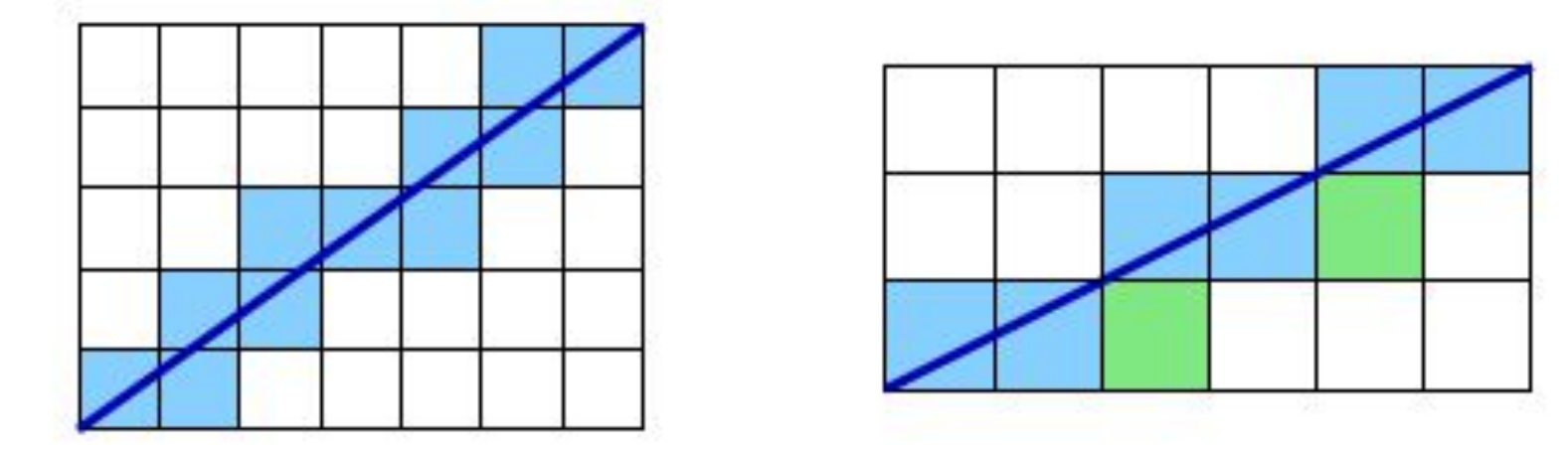}
\label{fig:bd}
\end{figure}
\end{example}

We can now conjecture which connected skew shapes are maximal in $\mathcal{P}_N$. This conjecture is due to McNamara and Pylyavskyy and is stated in \cite[Conjecture 1.3]{maxsupport}.

\begin{conjecture}\label{conj:maxelP} \cite[Conjecture 1.3]{maxsupport} In the subposet of $\mathcal{P}_N$ consisting of connected skew shapes, there are exactly $N$ maximal elements given by $P_{R, N-R+1}$ for $R=1, \ldots , N$.
\end{conjecture}

As we will show in Corollary~\ref{cor:eqrib}, $\bd$ is an equitable ribbon and hence we obtain the following conjecture for $\mathcal{R}((a+1)^na^m)$, which we prove in certain cases in Theorem~\ref{thm:chains} and Corollary~\ref{cor:shortends}, and confirm up to uniqueness in Corollary~\ref{cor:even}.

\begin{conjecture}\label{conj:maxel} In the poset $\mathcal{R}((a+1)^na^m)$ there is exactly one maximal element given by
$$P_{n+m, (n+m)a-m+1}.$$
\end{conjecture}

In \cite[Section 5.3]{maxsupport} they also conjectured a unique minimal element, as follows. 

\begin{conjecture}\label{conj:minel} \cite[Section 5.3]{maxsupport} In the poset $\mathcal{R}((a+1)^na^m)$ there is exactly one minimal element given by
$$(a+1)^{\lceil{\frac{n}{2}}\rceil} a^m (a+1)^{\lfloor{\frac{n}{2}}\rfloor}.$$
\end{conjecture}

We confirm this element is indeed minimal in Theorem~\ref{thm:minimal}.

\subsection{Box diagonal diagrams} \label{subsec:eqrib}
In order to better understand Conjecture~\ref{conj:maxel}, we present a new procedure to determine the box diagonal diagram. This will also verify that the box diagonal diagram is indeed an equitable ribbon. Because this construction relies on the Cartesian coordinate system, we adopt the  convention of counting rows from the bottom for this subsection \emph{only}. In this way, the top right corner of cell $(x,y)$ is at the point $(x,y)$.
\begin{lemma} \label{lem:eqrib} 
\hspace{2pt}
\begin{enumerate}
\item Suppose that $S\geq R$ and write $\frac{S}{R}=a-\epsilon$, where $a\in\mathbb{N}$ and $\epsilon\in[0,1)$. Then the box diagonal diagram $\bd$ is a ribbon and:\begin{itemize}
\item The $R$-th row has length $a$.
\item For $1\leq i\leq R-1$, the $i$-th row has length $a$ if $i\notin\{\lceil\frac{t}{1-\epsilon}\rceil:\hbox{ }t\in\mathbb{Z}\}$.
\item For $1\leq i\leq R-1$, the $i$-th row has length $(a+1)$ if $i\in\{\lceil\frac{t}{1-\epsilon}\rceil:\hbox{ }t\in\mathbb{Z}\}$.
In particular, these long rows occur precisely at $$\{\lceil\frac{t}{1-\epsilon}\rceil:\hbox{ }t\in\mathbb{Z}\}\cap\{1,\ldots,R-1\}.$$
\end{itemize}
\item Suppose that $R\geq S$ and write $\frac{R}{S}=b-\epsilon$, where $b\in\mathbb{N}$ and $\epsilon\in[0,1)$. Then the box diagonal diagram $\bd$ is a ribbon and:\begin{itemize}
\item The first column has length $b$.
\item For $2\leq j\leq S$, the $j$-th column has length $b$ if $j\notin\{\lfloor\frac{t}{1-\epsilon}\rfloor+1:\hbox{ }t\in\mathbb{Z}\}$.
\item For $2\leq j\leq S$, the $j$-th column has length $(b+1)$ if $j\in\{\lfloor\frac{t}{1-\epsilon}\rfloor+1:\hbox{ }t\in\mathbb{Z}\}$.
In particular, these long columns occur precisely at $$\{\lfloor\frac{t}{1-\epsilon}\rfloor+1:\hbox{ }t\in\mathbb{Z}\}\cap\{2,\ldots,S\}.$$
\end{itemize}  \end{enumerate} \end{lemma}

\begin{proof} Because the proofs of both parts are so similar, we only prove the first part. Consider the point $(x,y)$ on the line $L$ given in Definition~\ref{def:bd}.\\
 
Let $1\leq i\leq R$. We will determine which cells are in the $i$-th row of $\bd$.\\

If $x\leq\lfloor (i-1)\frac{S}{R}\rfloor$, then $y=\frac{R}{S}x\leq i-1$ so $L$ does not intersect the interior or top left corner of cell $(\lfloor(i-1)\frac{S}{R}\rfloor,i)$. Meanwhile, if $x=\lfloor(i-1)\frac{S}{R}\rfloor+1$, then $$(i-1)\frac{S}{R}<x\leq (i-1)\frac{S}{R}+1,\hbox{ so }i-1<y\leq i-1+\frac{R}{S}\leq i;$$ that is, $L$ intersects the right side of cell $(\lfloor(i-1)\frac{S}{R}\rfloor+1,i)$ at a point above the lower right corner. Since $L$ has positive slope, it therefore also intersects the interior of this cell. So the leftmost cell in the $i$-th row of $\bd$ is cell $$(\lfloor(i-1)\frac{S}{R}\rfloor+1,i).$$ 
The rightmost cell of the $R$-th row is cell $(S,R)$. For $1\leq i\leq R-1$, if $x\geq\lfloor i\frac{S}{R}\rfloor+1$, then $y=\frac{R}{S}x>i$ so $L$ does not intersect the interior or top left corner of cell $(\lfloor i\frac{S}{R}\rfloor+2,i)$. Meanwhile, if $x=\lfloor i\frac{S}{R}\rfloor$, then $$i\frac{S}{R}-1<x\leq i\frac{S}{R},\hbox{ so }i-1\leq i-\frac{R}{S}<y\leq i;$$ that is, $L$ intersects the right side of cell $(\lfloor i\frac{S}{R}\rfloor,i)$, which is the left side of cell $(\lfloor i\frac{S}{R}\rfloor+1,i)$. Either $L$ intersects this cell in the top left corner, or since $L$ has positive slope, in the interior of this cell. So the rightmost cell in the $i$-th row of $\bd$ is cell $$(\lfloor i\frac{S}{R}\rfloor+1,i)$$ if $1\leq i\leq R-1$, or cell $(S,R)$ if $i=R$. Note that for $1\leq i\leq R-1$, the rightmost cell of the $i$-th row is directly below the leftmost cell of the $(i+1)$-th row, so adjacent rows overlap in exactly one column and indeed $\bd$ is a ribbon.\\

Now the $R$-th row has $$S-(\lfloor(R-1)\frac{S}{R}\rfloor+1)+1=S-\lfloor S-a+\epsilon\rfloor-1+1=S-(S-a)=a$$ cells. For $1\leq i\leq R-1$, the $i$-th row has $$(\lfloor i\frac{S}{R}\rfloor+1)-(\lfloor(i-1)\frac{S}{R}\rfloor+1)+1=\lfloor i(a-\epsilon)\rfloor-\lfloor(i-1)(a-\epsilon)\rfloor+1=a+(\lfloor i(1-\epsilon)\rfloor-\lfloor(i-1)(1-\epsilon)\rfloor)$$ cells. Because $\lfloor i(1-\epsilon)\rfloor - \lfloor(i-1)(1-\epsilon)\rfloor$ is a nonnegative integer and $$\lfloor i(1-\epsilon)\rfloor -\lfloor(i-1)(1-\epsilon)\rfloor<i(1-\epsilon)-((i-1)(1-\epsilon)-1)=2-\epsilon\leq 2,$$ it must be $0$ or $1$. In particular, it is $1$ if and only if there is a $t\in\mathbb{Z}$ with $$(i-1)(1-\epsilon)< t\leq i(1-\epsilon);\hbox{ equivalently if }i=\lceil\frac{t}{1-\epsilon}\rceil.$$ Therefore, the $i$-th row has length $a$ if $i\notin\{\lceil\frac{t}{1-\epsilon}\rceil:\hbox{ }t\in\mathbb{Z}\}$ and has length $(a+1)$ if $i\in\{\lceil\frac{t}{1-\epsilon}\rceil:\hbox{ }t\in\mathbb{Z}\}$. \end{proof}

\begin{corollary} \label{cor:eqrib} For any $R$ and $S$, the box diagonal diagram $\bd$ is an equitable ribbon. \end{corollary}

\begin{proof} \begin{itemize}
\item If $S>R$, then by Part 1 of Lemma \ref{lem:eqrib}, the row lengths of $\bd$ are in $\{a,a+1\}$ where $a=\lceil\frac{S}{R}\rceil\geq 2$, so the column lengths of $\bd$ are in $\{1,2\}$.
\item If $R>S$, then by Part 2 of Lemma \ref{lem:eqrib}, the column lengths of $\bd$ are in $\{b,b+1\}$ where $b=\lceil\frac{R}{S}\rceil\geq 2$, so the row lengths of $\bd$ are in $\{1,2\}$.
\item If $S=R$, then by Part 1 of Lemma \ref{lem:eqrib}, the row lengths of $\bd$ are in $\{a,a+1\}$ for some $a$ and by Part 2 of Lemma \ref{lem:eqrib}, the column lengths of $\bd$ are in $\{b,b+1\}$ for some $b$.\end{itemize}\end{proof}

\begin{corollary} \label{cor:maxforposets} The box diagonal diagrams for the following posets are the following equitable ribbons. \begin{align}\label{eq:chain1a+1} &\mathcal{R}((a+1)^1a^m):&&a^{\lfloor\frac{m}{2}\rfloor}(a+1)a^{\lceil\frac{m}{2}\rceil}\\\label{eq:chain1a} &\mathcal{R}((a+1)^na^1):&&(a+1)^na\\\label{eq:chain4} &\mathcal{R}((a+1)^2a^2):&&a(a+1)(a+1)a\\\label{eq:chain5} &\mathcal{R}((a+1)^2a^3):&&a(a+1)a(a+1)a\\\label{eq:shortendscor} &\mathcal{R}((a+1)^na^2):&&a(a+1)^na\\\label{eq:smallscor} &\mathcal{R}((a+1)^na^{d(n+1)}):&&a^d(a+1)a^d\cdots a^d(a+1)a^d\end{align}\end{corollary}

\begin{remark} \label{rem:maxforposets} Later we will prove in each case that these box diagonal diagrams are indeed maximal elements of their respective posets, thus confirming Conjecture~\ref{conj:maxel} in these cases. We prove the first four in Theorem \ref{thm:chains} of Section \ref{sec:chains}, the fifth in Corollary \ref{cor:shortends} of Section \ref{sec:shortends}, and the sixth in Corollary \ref{cor:even} of Section \ref{sec:coarsenings}. \end{remark}

\begin{proof} We prove only \eqref{eq:chain5} because it is the most representative of how each part can be proved. We  use Lemma \ref{lem:eqrib} to calculate $P_{R,S}$. Because the cases of $a\geq 2$ and $a=1$ are qualitatively different, we handle them separately.\\

If $a\geq 2$, then \begin{itemize}\item $R=5$ \item $N=5a+2$ \item $S=5a-2$ \item $\frac{S}{R}=a-\frac{2}{5}=a-\epsilon$ where $\frac{1}{1-\epsilon}=\frac{5}{3}$ \item so by Part 1 of Lemma \ref{lem:eqrib} the long rows of $P_{R,S}$ are at $$\{\lceil\frac{5}{3}t\rceil:\hbox{ }t\in\mathbb{Z}\}\cap\{1,\ldots,4\}=\{2,4\}.$$ \end{itemize} So $P_{R,S}=a(a+1)a(a+1)a$.\\

If $a=1$, then \begin{itemize}\item $R=5$ \item $N=7$ \item $S=3$ \item $\frac{R}{S}=2-\frac{1}{3}=b-\epsilon$, where $\frac{1}{1-\epsilon}=\frac{3}{2}$ \item so by Part 2 of Lemma \ref{lem:eqrib}, the long columns of $P_{R,S}$ are at $$\{\lfloor\frac{3}{2}t\rfloor+1\}\cap\{2,\ldots,3\}=\{2\}.$$ \end{itemize} So $P_{R,S}$ has column lengths $232$, and hence $P_{R,S}$ has row lengths $12121$.\end{proof}

\section{Ribbon operations for Schur-positivity}\label{sec:tools}

In the next two subsections, we establish two ribbon Schur function inequalities, which will be pivotal for calculations in Section \ref{sec:chains}.

\subsection{Moving a cell from the first row}\label{subsec:moving1}

Our first inequality compares a ribbon $\alpha$ to a ribbon $\beta$ that is obtained from $\alpha$ by moving a cell from the first row to a different row.

\begin{definition} \label{def:moving1} Let $\alpha=\alpha_1\cdots\alpha_R$ be a composition with $\alpha _1 \geq 2$ and $2\leq i\leq R$. Then define $$M_i(\alpha)=(\alpha_1-1)\alpha_2\cdots\alpha_{i-1}(\alpha_i+1)\alpha_{i+1}\cdots\alpha_R;$$ that is, $M_i(\alpha)$ is the composition formed by decrementing the first part of $\alpha$ by one and incrementing the $i$-th part of $\alpha$ by one. \end{definition}

\begin{example} \label{ex:moving1first} Let $\alpha=9544$ and $i=3$. Then  $M_3(\alpha)=8554$.\end{example}

\begin{definition} \label{def:restrictedLR} Let $\alpha=\alpha_1\cdots\alpha_R$ be a composition, $2\leq i\leq R$, and $\nu$ be a partition. Then we define $C'_{\alpha,i,\nu}$ to be the set of LR tableaux $T$ such that\begin{itemize}
\item $sh(T)=\alpha$, $c(T)=\nu$, and
\item if $T_{i,i_1}=1$, then $i\leq R-1$ and $T_{i,i_1+1}\geq T_{i+1,i_1}$,

\end{itemize} where the first cell in the $i$-th row of $T$ is in column $i_1$. We refer to such $T$ as \emph{row-$i$-restricted LR tableaux}, or when the context is clear, simply as \emph{restricted LR tableaux}. Also set $$c'_{\alpha,i,\nu}=|C'_{\alpha,i,\nu}|\hbox{ and }c'_{\alpha,i,j,\nu}=|C'_{\alpha,i,\nu}\cap C'_{\alpha,j,\nu}|.$$
\end{definition}

Informally, we can illustrate a row-$i$-restricted LR tableau in the following way. If the first cell of row $i$ denoted by the asterisk is a $1$ then the bullet to the right must be at least the bullet below.

$$\tableau{&&\ast&\bullet&\ &\ \\ \ &\ &\bullet}$$

\begin{example} \label{ex:lrrestricted} The following LR tableau $T$ is of shape $\alpha=8554$ and has content $\nu=11\hbox{ }74$.
$$T=\tableau{&&&&&&&&&&&1&1&1&1&1&1&1&1\\&&&&&&&2&2&2&2&2\\&&&1&3&3&3&3\\1&1&2&2}$$Because $T_{2,8}\neq 1$, we have $T\in C'_{\alpha,2,\nu}$. Since $T_{3,4}=1$, plus $i=3\leq 3=R-1$ and $$T_{3,5}=3\geq 2=T_{4,4},$$ we have $T\in C'_{\alpha,3,\nu}$ as well. However $T\notin C'_{\alpha,4,\nu}$ because $T_{4,1}=1$ but now $i=4\nleq 3=R-1$.\end{example}

\begin{theorem} \label{thm:moving1} Let $\alpha=\alpha_1\cdots\alpha_R$ be a composition, $2\leq i\leq R$, and $\beta=M_i(\alpha)$. If $\beta_1\geq \beta_2+\cdots+\beta_i-i+1$, then $$\rb\geq_s \ra;$$ moreover, $$\rb-\ra=\sum_\nu c'_{\beta,i,\nu}s_\nu.$$ \end{theorem}

\begin{example} \label{ex:moving1} Let $\alpha=9544$ and $i=3$, so that $\beta=8554$ by Example~\ref{ex:moving1first}. Indeed, $8\geq 5+5-3+1$, so by Theorem \ref{thm:moving1} we can conclude that $\rb\geq_s \ra$. Furthermore, the difference $\rb-\ra$ is given by restricted LR tableaux $T$ of shape $\beta=8544$ where $$\hbox{if } T_{3,4}=1,\hbox{ then }T_{3,5}\geq T_{4,4}.$$  We illustrate this restriction as follows, where if the asterisk is a $1$ then the bullet to the right must be at least the bullet below. $$\tableau{&&&&&&&&&&&\hspace{0pt}&\hspace{0pt}&\hspace{0pt}&\hspace{0pt}&\hspace{0pt}&\hspace{0pt}&\hspace{0pt}&\hspace{0pt}\\&&&&&&&\hspace{0pt}&\hspace{0pt}&\hspace{0pt} & \hspace{0pt}& \hspace{0pt}\\&&&\ast&\bullet&\hspace{0pt}&\hspace{0pt}&\hspace{0pt}\\\hspace{0pt}&\hspace{0pt}&\hspace{0pt}&\bullet}$$\end{example}

\begin{proof} Let $A$ be the set of LR tableaux of shape $\alpha$ and let $B$ be the set of LR tableaux of shape $\beta$. We will prove the desired inequality by finding an injective content-preserving function $$f_i:A\rightarrow B$$ and we will calculate the difference $\rb-\ra$ by determining which LR tableaux of $B$ are not in the image of $A$.\\

For an LR tableau $T\in A$, set $f_i(T)$ to be the tableau of shape $\beta$ where the top row is filled with all $1$'s, the $i$-th row is filled with a $1$ followed by the entries of the $i$-th row in $T$, and all other rows are filled as in $T$. Informally, we can think of $f_i$ as moving a $1$ from the top row to the $i$-th row.\\

By construction, $f_i(T)$ is an SSYT of shape $\beta$ with the same content as $T$. Also note that $f_i$ is injective because we can reconstruct $T$ from $f_i(T)$. It remains to verify the lattice word condition for $f_i(T)$. Because the reading word of $f_i(T)$ differs from that of $T$, which is a lattice word, only by reading a $1$ after the other entries of the $i$-th row, it suffices to check that there have not been too many $2$'s read before this point. The tableau $f_i(T)$ has $\beta_1$ $1$'s in the first row. Then in the subtableau $U$ of $f_i(T)$ supported on the following $(i-1)$ rows, with the leftmost cell of row $i$ removed, since we know that cell will be filled with a $1$, we have that \begin{align}\nonumber\hbox{the number of $2$'s in $U$}&\leq\hbox{the number of columns of $U$}\\\nonumber&=\hbox{the number of cells of $U$}-\hbox{the number of rows of $U$}+1\\\nonumber&=\beta_2+\cdots+\beta_{i-1}+(\beta_i-1)-(i-1)+1\\\nonumber&=\beta_2+\cdots+\beta_i-i+1\\\nonumber&\leq\beta_1\end{align}
by hypothesis. Therefore the lattice word condition is preserved and indeed $f_i(T)\in B$.\\

Finally, consider an LR tableau $f_i(T)$ in the image of $f_i$. Let $i_1$ be the column of the first cell in the $i$-th row of $f_i(T)$. By construction, we know that $f_i(T)_{i,i_1}=1$. Additionally, if $i\leq R-1$, we have $$f_i(T)_{i,i_1+1}=T_{i,i_1}< T_{i,i_1+1}=f_i(T)_{i,i_1+1}$$ because $T$ is column strict. Therefore, the LR tableaux of $B$ that are not in the image of $f_i$ are precisely the row-$i$-restricted LR tableaux of shape $\beta$ and so $$\rb-\ra=\sum_{T\in  B}s_{c(T)}-\sum_{T'\in A}s_{c(T')}=\sum_{T\in (B\setminus f_i(A))}s_{c(T)}=\sum_{\nu}\sum_{T\in C'_{\beta,i,\nu}}s_{c(T)}=\sum_\nu c'_{\beta,i,\nu}s_\nu,$$ as desired.\end{proof}

\begin{example} \label{ex:moving1proof} If $T$ is the tableau below, then $f_4(T)$ is as indicated.
\begin{align}\nonumber T&=\tableau{&&&&&&&&&&&1&1&1&1&1&1&1&1&1\\&&&&&&&2&2&2&2&2\\&&&1&3&3&3&3\\&1&2&2}\\\nonumber f_4(T)&=\tableau{&&&&&&&&&&&1&1&1&1&1&1&1&1\\&&&&&&&2&2&2&2&2\\&&&1&3&3&3&3\\1&1&2&2}\end{align} Note that $f_4(T)$ can not be in the image of $f_2$ because $f_4(T)_{2,8}\neq 1$. Also, $f_4(T)$ can not be in the image of $f_3$ because although $f_3(T)_{3,4}=1$, its preimage would have had to be $$\tableau{&&&&&&&&&&&1&1&1&1&1&1&1&1&1\\&&&&&&&2&2&2&2&2\\&&&&3&3&3&3\\&&1&2&2}$$that is not column strict.\end{example} 

We now make a slight generalization of Theorem \ref{thm:moving1}.

\begin{theorem} \label{thm:moving1cor} Let $\alpha=\alpha_1\cdots\alpha_R$ be a composition, $2\leq i<j\leq R$, and $\beta=M_i(M_j(\alpha))$. If $\beta_1\geq \beta_2+\cdots+\beta_j-j+1$, then \begin{equation} \label{eq:moving1cor}\rb-r_{M_j(\alpha)}-r_{M_i(\alpha)}+\ra =\sum_\nu c'_{\beta,i,j,\nu}s_\nu \geq_s 0.\end{equation}\end{theorem}

\begin{proof} To prove this, we will use three applications of Theorem \ref{thm:moving1}. First, we can check that $$\beta_1\geq\beta_2+\cdots+\beta_i+(\beta_{i+1}+\cdots+\beta_j)-j+1\geq\beta_2+\cdots+\beta_i+(j-i)-j+1=\beta_2+\cdots+\beta_i-i+1,$$ and so by Theorem \ref{thm:moving1} the first two terms of Equation \eqref{eq:moving1cor} give us $$\rb-r_{M_j(\alpha)}=\sum_\nu c'_{\beta,i,\nu}s_\nu.$$ Second, because the first part of $M_i(\alpha)$ is one larger than the first part of $\beta$, and otherwise the two compositions agree to at least the first $i$ parts, the hypothesis of Theorem \ref{thm:moving1} is satisfied for $M_i(\alpha)$ and so the latter two terms of Equation \eqref{eq:moving1cor} give us $$-(r_{M_i(\alpha)}-\ra)=-\sum_\nu c'_{M_i(\alpha),i,\nu}s_\nu.$$ Third, letting $A$ be the set of row-$i$-restricted LR tableaux of shape $M_i(\alpha)$ and $B$ be the set of row-$i$-restricted LR tableaux of shape $\beta$, then we can apply the same injection $f_j:A\rightarrow B$ as in the proof of Theorem \ref{thm:moving1}, noting that it preserves the restriction on row $i$. The restricted LR tableaux of $B$ not in the image of $f_j$ will now have the restriction on both row $i$ and row $j$, and so $$\rb-r_{M_j(\alpha)}-r_{M_i(\alpha)}+\ra=\sum_\nu c'_{\beta,i,\nu}s_\nu-\sum_\nu c'_{M_i(\alpha),i,\nu}s_\nu=\sum_\nu c'_{\beta,i,j,\nu}s_\nu\geq_s 0.$$\end{proof}

\begin{example} \label{ex:moving1cor} Let $\alpha=10\hbox{ }444$, $i=2$, and $j=3$, so that $\beta=8554$. Indeed, $8\geq 5+5-3+1$, so by Theorem \ref{thm:moving1cor} the difference $r_{8554}-r_{9544}-r_{9454}+r_{10\hbox{ }444}$ is Schur-positive and is illustrated by the diagram below.  $$\tableau{&&&&&&&&&&&\hspace{0pt}&\hspace{0pt}&\hspace{0pt}&\hspace{0pt}&\hspace{0pt}&\hspace{0pt}&\hspace{0pt}&\hspace{0pt}\\&&&&&&&\ast&\bullet&\hspace{0pt} & \hspace{0pt}& \hspace{0pt}\\&&&\ast&\bullet&\hspace{0pt}&\hspace{0pt}&\bullet\\\hspace{0pt}&\hspace{0pt}&\hspace{0pt}&\bullet}$$\end{example}

\begin{remark} \label{rem:moving1cor} Theorem \ref{thm:moving1cor} generalizes to a Schur-positive inclusion-exclusion-type ribbon Schur function sum where cells have been moved to more than two rows. However, Theorem \ref{thm:moving1cor} is sufficient for our results in Section \ref{sec:chains}.\end{remark}

\subsection{Exchanging adjacent rows}\label{subsec:celltrans}
 
Our second inequality compares a ribbon $\alpha$ to a ribbon $\beta$ that is obtained from $\alpha$ by exchanging adjacent rows of $\alpha$. We will need to develop some machinery before we make the comparison.

\begin{lemma} \label{lem:connearcon} \cite[Equation (2.2)]{btvw} Let $\alpha$ and $\beta$ be compositions. Then the ribbon Schur functions $\ra$ and $\rb$ satisfy $$\ra \rb=r_{\alpha\cdot\beta}+r_{\alpha\odot\beta}.$$ \end{lemma}
 
\begin{definition} \label{def:uniint} For two partitions $\lambda=\lambda_1\lambda_2\cdots\lambda_{\ell(\lambda)}$ and $\mu=\mu_1\mu_2\cdots\mu_{\ell(\mu)}$ where $\ell(\lambda)\geq\ell(\mu)$ without loss of generality, define the \emph{union of $\lambda$ and $\mu$} to be the partition $$\lambda\vee\mu=\max(\lambda_1,\mu_1)\max(\lambda_2,\mu_2)\cdots\max(\lambda_{\ell(\lambda)},\mu_{\ell(\lambda)})$$ and the \emph{intersection of $\lambda$ and $\mu$} to be the partition $$\lambda\wedge\mu=\min(\lambda_1,\mu_1)\min(\lambda_2,\mu_2)\cdots\min(\lambda_{\ell(\mu)},\mu_{\ell(\mu)}).$$ These operations are indeed the set-theoretic union and intersection of the corresponding diagrams. For two skew shapes $\lambda/\mu$ and $\nu/\rho$, define the union and intersection to be the skew shapes $$(\lambda/\mu)\vee(\nu/\rho)=(\lambda\vee\nu)/(\mu\vee\rho)\hbox{ and }(\lambda/\mu)\wedge(\nu/\rho)=(\lambda\wedge\nu)/(\mu\wedge\rho).$$
  We warn that these are the skew shapes arising from the set-theoretic union and intersection of the outer and inner shapes separately, and may not be the set-theoretic union and intersection of the overall diagram. \end{definition}
  
\begin{example}  \label{ex:lpppart1} Let $\lambda=52$, $\mu=3$, $\nu=53$, and $\rho=2$. Then $\lambda/\mu$ is the ribbon $22$, $\nu/\rho$ is the ribbon $33$, \begin{align}\nonumber (\lambda/\mu)\vee(\nu/\rho)&=(\lambda\vee\nu)/(\mu\vee\rho)=53/3\hbox{ is the ribbon $23$, and }\\\nonumber (\lambda/\mu)\wedge(\nu/\rho)&=(\lambda\wedge\nu)/(\nu\wedge\rho)=52/2\hbox{ is the ribbon $32$}.\end{align}\end{example}
  
  \begin{theorem}\label{thm:lpp} \cite[Theorem 5]{lpp} Let $\lambda/\mu$ and $\nu/\rho$ be any two skew shapes. Then we have $$s_{(\lambda/\mu)\vee(\nu/\rho)}s_{(\lambda/\mu)\wedge(\nu/\rho)}
  \geq_s s_{\lambda/\mu}s_{\nu/\rho}.$$ \end{theorem}
  
  \begin{example} \label{ex:lpp}  Let $\lambda=52$, $\mu=3$, $\nu=53$, and $\rho=2$. Then by Example~\ref{ex:lpppart1} and Theorem \ref{thm:lpp} we have $$r_{23}r_{32}\geq_sr_{22}r_{33}.$$\end{example}

We can now state the main result of this subsection.
 
  \begin{theorem} \label{thm:celltr} Let $\delta$ and $\gamma$ be compositions and let $a'\geq a$. Suppose that either of the following hold:
  \begin{enumerate}
  \item $\ell(\gamma)\geq\ell(\delta)\geq 0$ and $\delta^*$ dominates $\gamma$  up to length $\ell(\delta)$, that is, $$\delta^*_1+\cdots+\delta^*_i\geq\gamma_1+\cdots+\gamma_i\hbox{ for }1\leq i\leq\ell(\delta);\hbox{ or }$$
  \item $0<\ell(\gamma)<\ell(\delta)$ and $\delta^*$ dominates $\gamma$ up to length $(\ell(\gamma)-1)$ and strictly at $\ell(\gamma)$, that is, $$\delta^*_1+\cdots+\delta^*_i\geq\gamma_1+\cdots+\gamma_i\hbox{ for }1\leq i\leq\ell(\gamma)-1\hbox{ and }\delta^*_1+\cdots+\delta^*_{\ell(\gamma)}>\gamma_1+\cdots+\gamma_{\ell(\gamma)}.$$
  \end{enumerate}
  Then $$r_{\delta aa'\gamma}\geq_s r_{\delta a'a\gamma}.$$ \end{theorem}
  
  \begin{example} \label{ex:celltr} \begin{enumerate}
  \item Let $\delta=45$, $\gamma=44$, $a=4$, and $a'=5$. Indeed $\ell(\gamma)\geq\ell(\delta)$ and $\delta^*=54$ dominates $\gamma=44$ up to row $2$. Therefore
  $$r_{45\hbox{ }45\hbox{ }44}\geq_s r_{45\hbox{ }54\hbox{ }44}.$$
  \item Let $\delta=45$, $\gamma=4$, $a=4$, and $a'=5$. Indeed $0<\ell(\gamma)<\ell(\delta)$ and $\delta^*=54$ dominates $\gamma=4$ strictly at row 1. Therefore
  $$r_{45\hbox{ }45\hbox{ }4}\geq_s r_{45\hbox{ }54\hbox{ }4}.$$
  \end{enumerate}\end{example}
  
  Before presenting the proof, as a concrete illustration of it, we work through what happens in these two cases.
  
  \begin{enumerate}
  \item By Lemma \ref{lem:connearcon}, we have that \begin{align}\nonumber r_{454}r_{544}&=r_{454544}+r_{45944}\hbox{ and }\\\nonumber r_{455}r_{444}&=r_{455444}+r_{45944}\end{align} so it suffices to show that \begin{equation}\label{eq:case1ex}r_{454}r_{544}\geq_s r_{455}r_{444}=r_{554}r_{444}.\end{equation}
  By Theorem \ref{thm:lpp}, it remains to express the ribbons $554$ and $444$ on the right side of Equation \eqref{eq:case1ex} as skew shapes $\lambda/\mu$ and $\nu/\rho$ such that ribbons on the left side of Equation \eqref{eq:case1ex} will be their union and intersection.\\
   
The skew shapes will be as follows.
   \begin{align*}\lambda/\mu=a'\delta^*=&\ \tableau{\times&\times&\times&\times&\times&\times&\times&\star&\hspace{0pt}&\hspace{0pt}&\hspace{0pt}&\hspace{0pt}\\\times&\times&\times&\hspace{0pt}&\hspace{0pt}&\hspace{0pt}&\hspace{0pt}&\hspace{0pt}&\\\hspace{0pt}&\hspace{0pt}&\hspace{0pt}&\hspace{0pt}}\\
  \nu/\rho=a\gamma=&\ \tableau{\times&\times&\times&\times&\times&\times&\times&\star&\hspace{0pt}&\hspace{0pt}&\hspace{0pt}\\\times&\times&\times&\times&\hspace{0pt}&\hspace{0pt}&\hspace{0pt}&\hspace{0pt}&\\\times&\hspace{0pt}&\hspace{0pt}&\hspace{0pt}&\hspace{0pt}\\}\end{align*}

Informally, an empty column was added to $\nu/\rho$ to align the starred cells. Then the dominance condition tells us that the $\lambda/\mu$ shape, corresponding to $a'\delta^*$, is more to the left, which causes the union to be $\gamma$ with the longer row $a'$ at the top, and the intersection to be $\delta^*$ with the shorter row $a$ at the top.\\

To be precise, letting $\lambda=12\hbox{ }84$, $\mu=73$, $\nu=11\hbox{ }85$, and $\rho=741$, we have that $\lambda/\mu$ is the ribbon $554=a'\delta^*$ and $\nu/\rho$ is the ribbon $444=a\gamma$. Additionally,
        \begin{align}\nonumber (\lambda/\mu)\vee(\nu/\rho)&=(12\hbox{ }84\vee 11\hbox{ }85)/(73\vee741)=12\hbox{ }85/741\hbox{ is the ribbon }544=a'\gamma\hbox{ and }\\\nonumber(\lambda/\mu)\wedge(\nu/\rho)&=(12\hbox{ }84\wedge 11\hbox{ }85)/(73\wedge741)=11\hbox{ }84/73\hbox{ is the ribbon }454=a\delta^*,\end{align} as desired.\\
    
    \item The second case is similar. By Lemma \ref{lem:connearcon}, we have that \begin{align}\nonumber r_{454}r_{54}&=r_{45454}+r_{4594}\hbox{ and }\\\nonumber r_{455}r_{44}&=r_{45544}+r_{4594}\end{align} so it suffices to show that \begin{equation}\label{eq:case2ex} r_{454}r_{54}\geq_s r_{455}r_{44}=r_{554}r_{44}.\end{equation}
    By Theorem \ref{thm:lpp}, it remains to express the ribbons $554$ and $44$ on the right side of Equation \eqref{eq:case2ex} as skew shapes $\lambda/\mu$ and $\nu/\rho$ such that the ribbons on the left side of Equation \eqref{eq:case2ex} will be their union and intersection.\\
    
     The skew shapes will be as follows.
     \begin{align*}\lambda/\mu=a'\delta^*=&\ \tableau{\times&\times&\times&\times&\times&\times&\times&\star&\hspace{0pt}&\hspace{0pt}&\hspace{0pt}&\hspace{0pt}\\\times&\times&\times&\hspace{0pt}&\hspace{0pt}&\hspace{0pt}&\hspace{0pt}&\hspace{0pt}&\\\hspace{0pt}&\hspace{0pt}&\hspace{0pt}&\hspace{0pt}}\\
    \nu/\rho=a\gamma=&\ \tableau{\times&\times&\times&\times&\times&\times&\times&\star&\hspace{0pt}&\hspace{0pt}&\hspace{0pt}\\\times&\times&\times&\times&\hspace{0pt}&\hspace{0pt}&\hspace{0pt}&\hspace{0pt}&\\\times&\times&\times&\times}\end{align*}

Informally, four empty columns were added to $\nu/\rho$ to align the starred cells. The strict dominance condition also allowed us to add an empty row to $\nu/\rho$ to match the $\lambda/\mu$ shape. This causes the bottom row of the $\lambda/\mu$ shape, corresponding to $a'\delta^*$, to disappear in the union, producing $\gamma$ with the longer row $a'$ at the top; and the bottom row of the $\nu/\rho$ shape to appear in the intersection, producing $\delta^*$ with the shorter row $a$ at the top.\\

To be precise, letting $\lambda=12\hbox{ }84$, $\mu=73$, $\nu=11\hbox{ }84$, and $\rho=744$, we have that $\lambda/\mu$ is the ribbon $554=a'\delta^*$ and $\nu/\rho$ is the ribbon $44=a\gamma$. Additionally,
          \begin{align}\nonumber(\lambda/\mu)\vee(\nu/\rho)&=(12\hbox{ }84\vee 11\hbox{ }84)/(73\vee 744)=12\hbox{ }84/744\hbox{ is the ribbon }54=a'\gamma\hbox{ and }\\\nonumber (\lambda/\mu)\wedge(\nu/\rho)&=(12\hbox{ }84\wedge 11\hbox{ }84)/(73\wedge744)=11\hbox{ }84/73\hbox{ is the ribbon }454=a\delta^*,\end{align} as desired. \\
  \end{enumerate}
  
  Now we present the proof of Theorem \ref{thm:celltr}. 
  
  \begin{proof} By Lemma \ref{lem:connearcon}, we have \begin{align}\nonumber r_{\delta a}r_{a'\gamma}&=r_{\delta aa'\gamma}+r_{\delta(a+a')\gamma}\hbox{ and }\\\nonumber r_{\delta a'}r_{a\gamma}&=r_{\delta a'a\gamma}+r_{\delta(a+a')\gamma},\end{align} so it suffices to show that  
  \begin{equation}\label{eq:case1}r_{a\delta^*}r_{a'\gamma}=r_{\delta a}r_{a'\gamma}\geq_s r_{\delta  a'}r_{a\gamma}=r_{a'\delta^*}r_{a\gamma}.\end{equation} By Theorem \ref{thm:lpp}, it remains to express the ribbons $a'\delta^*$ and $a\gamma$ on the right side of Equation \eqref{eq:case1} as skew shapes $\lambda/\mu$ and $\nu/\rho$ so that the ribbons on the left side of Equation \eqref{eq:case1} will be their union and intersection.\\
  
  \begin{enumerate}
  \item Let $M=|\delta|-\ell(\delta)$ and define partitions
   $\lambda$, $\mu$, $\nu$, and $\rho$ as follows. \begin{itemize}
\item $\lambda_1=M+a'$
\item$\lambda_i=M-(\delta^*_1+\cdots+\delta^*_{i-2})+(i-1)$ for $2\leq i\leq \ell(\delta)+1$\\
\item$\mu_i=M-(\delta^*_1+\cdots+\delta^*_{i-1})+(i-1)$ for $1\leq i\leq\ell(\delta)$\\
\item $\nu_1=M+a$
\item $\nu_i=M-(\gamma_1+\cdots+\gamma_{i-2})+(i-1)$ for $2\leq i\leq\ell(\gamma)+1$\\
\item $\rho_i=M-(\gamma_1+\cdots+\gamma_{i-1})+(i-1)$ for $1\leq i\leq \ell(\gamma)$\\
\end{itemize}

Now $\lambda/\mu$ is the ribbon $a'\delta^*$ and $\nu/\rho$ is the ribbon $a\gamma$. Also, because $\delta^*$ dominates $\gamma$ up to length $\ell(\delta)$, we have that \begin{itemize}\item $\nu_i\geq\lambda_i$ for $2\leq i\leq\ell(\gamma)+1$ and \item $\rho_i\geq\mu_i$ for $1\leq i\leq \ell(\gamma)$.\end{itemize} Therefore $$(\lambda/\mu)\vee(\nu/\rho)=(\lambda\vee\nu)/(\mu\vee\rho)=\hat{\nu}/\rho\hbox{ is the ribbon }a'\gamma,$$ where $\hat{\nu} = (M+a')\nu_2\cdots\nu_{\ell(\nu)}$, and similarly
   $$(\lambda/\mu)\wedge(\nu/\rho)=(\lambda\wedge\nu)/(\mu\wedge\rho)=\check{\lambda}/\mu\hbox{ is the ribbon }a\delta^*,$$ where $\check{\lambda}=(M+a)\lambda_2\cdots\lambda_{\ell(\lambda)}$.\\
   
  \item Let $M=|\delta|-\ell(\delta)$ and define partitions $\lambda$, $\mu$, $\nu$, and $\rho$ as follows.
\begin{itemize}
\item $\lambda_1=M+a'$
\item$\lambda_i=M-(\delta^*_1+\cdots+\delta^*_{i-2})+(i-1)$ for $2\leq i\leq \ell(\delta)+1$\\
\item$\mu_i=M-(\delta^*_1+\cdots+\delta^*_{i-1})+(i-1)$ for $1\leq i\leq\ell(\delta)$\\
\item $\nu_1=M+a$
\item $\nu_i=M-(\gamma_1+\cdots+\gamma_{i-2})+(i-1)$ for $2\leq i\leq\ell(\gamma)+1$
\item $\nu_i=M-(\delta^*_1+\cdots+\delta^*_{i-2})+(i-1)$ for $\ell(\gamma)+2\leq i\leq \ell(\delta)+1$\\
\item $\rho_i=M-(\gamma_1+\cdots+\gamma_{i-1})+(i-1)$ for $1\leq i\leq \ell(\gamma)+1$
\item $\rho_i=M-(\delta^*_1+\cdots+\delta^*_{i-2})+(i-1)$ for $\ell(\gamma)+2\leq i\leq \ell(\delta)+1$\\
\end{itemize}

With this definition it is not clear that $\nu$ and $\rho$ are partitions. However, because $\delta^*$ strictly dominates $\gamma$ at length $\ell(\gamma)$, we have that \begin{align}\nonumber \nu_{\ell(\gamma)+1}\geq\rho_{\ell(\gamma)+1}&=M-(\gamma_1+\cdots+\gamma_{\ell(\gamma)})+\ell(\gamma)\\\nonumber &\geq M-(\delta^*_1+\cdots+\delta^*_{\ell(\gamma)})+(\ell(\gamma)+1)=\rho_{\ell(\gamma)+2}=\nu_{\ell(\gamma)+2},\end{align} so indeed $\nu$ and $\rho$ are partitions.\\

Now $\lambda/\mu$ is the ribbon $a'\delta^*$ and $\nu/\rho$ is the ribbon $a\gamma$. Also, because $\delta^*$ dominates $\gamma$ up to length $\ell(\gamma)$, and also because $\lambda_i=\nu_i=\rho_i$ for $\ell(\gamma)+2\leq i\leq \ell(\delta)+1$ by definition, we have that \begin{itemize}
\item $\nu_i\geq\lambda_i$ for $2\leq i\leq\ell(\delta)+1$ and 
\item $\rho_i\geq\mu_i$ for $1\leq i\leq \ell(\delta)+1$.\end{itemize} Therefore $$(\lambda/\mu)\vee(\nu/\rho)=(\lambda\vee\nu)/(\mu\vee\rho)=\hat{\nu}/\rho\hbox{ is the ribbon }a'\gamma,$$ where $\hat{\nu} = (M+a')\nu_2\cdots\nu_{\ell(\nu)}$, and similarly
   $$(\lambda/\mu)\wedge(\nu/\rho)=(\lambda\wedge\nu)/(\mu\wedge\rho)=\check{\lambda}/\mu\hbox{ is the ribbon }a\delta^*,$$ where $\check{\lambda}=(M+a)\lambda_2\cdots\lambda_{\ell(\lambda)}$. This completes the proof.\end{enumerate}\end{proof}

\section{Totally ordered equitable ribbons}\label{sec:chains}

Our goal for this section is to prove the following theorem, which identifies cases in which $\mathcal{R}((a+1)^na^m)$ is a chain. We will see in Corollary \ref{cor:incompar} of Section \ref{sec:coarsenings} that these are in fact the \emph{only} cases in which $\mathcal{R}((a+1)^na^m)$ is a chain.

\begin{theorem} \label{thm:chains} The partially ordered set $\mathcal{R}((a+1)^na^m)$ of ribbons with $n$ rows of length $(a+1)$ and $m$ rows of length $a$ is a chain in the following cases.\begin{align} \label{chain:1a+1} \mathcal{R}((a+1)^1a^m)&=a^{\lfloor\frac{m}{2}\rfloor}(a+1)a^{\lceil\frac{m}{2}\rceil}>_s\cdots>_sa(a+1)a^{m-1}>_s(a+1)a^m\\\label{chain:1a} \mathcal{R}((a+1)^na^1)&=a(a+1)^n>_s(a+1)a(a+1)^{n-1}>_s\cdots>_s(a+1)^{\lfloor\frac{n}{2}\rfloor} a(a+1)^{\lceil\frac{n}{2}\rceil}\\\label{chain:4}\mathcal{R}((a+1)^2a^2)&=a(a+1)(a+1)a\stackrel{1}{>_s}(a+1)a(a+1)a\stackrel{2}{>_s}(a+1)(a+1)aa\\\nonumber&\stackrel{3}{>_s}(a+1)aa(a+1)\\\label{chain:5}\mathcal{R}((a+1)^2a^3)&=a(a+1)a(a+1)a\stackrel{1}{>_s}a(a+1)(a+1)aa\stackrel{2}{>_s}(a+1)a(a+1)aa\\\nonumber&\stackrel{3}{>_s}(a+1)aa(a+1)a\stackrel{4}{>_s}(a+1)(a+1)aaa\stackrel{5}{>_s}(a+1)aaa(a+1)\end{align} \end{theorem}

\begin{remark} \label{rem:chainsmaxforposets} According to \eqref{eq:chain1a+1} through \eqref{eq:chain5} of Corollary \ref{cor:maxforposets}, Theorem \ref{thm:chains} confirms that for each of the four posets above, the box diagonal diagram is indeed the unique maximal element, confirming Conjecture~\ref{conj:maxel} in these cases, which turn out to be all chains. Theorem \ref{thm:chains} also confirms Conjecture~\ref{conj:minel} in these cases. \end{remark}

 We first establish the strictness of the inequalities above.

\begin{lemma} \label{lem:strict} The inequalities in Theorem \ref{thm:chains} are all strict. \end{lemma}

\begin{proof} First note that if $a\geq 2$, then all of the ribbons above are equitable, so by Corollary~\ref{cor:ineq} could only be equal to their reverse, but since none of the ribbons are reverses of each other the result follows. Therefore it remains to consider the $a=1$ case. We will show that in each chain, any element $\alpha$ can not factor as a composition $\beta\circ\gamma$ of two nonsymmetric factors, and therefore again the equivalence class of $\alpha$ will only contain itself and its reverse.

\begin{enumerate}
\item Chain \eqref{chain:1a+1}: If $\beta$ is nonsymmetric then it can not be all $1$'s, so there must be a near-concatenation of $\gamma$ within $\alpha$, giving rise to a part at least $2$. But this accounts for the single $2$ in $\alpha$, leaving $\gamma$ to be all $1$'s and not nonsymmetric.\\

\item Chain \eqref{chain:1a}: Note that if $n=0$ the result is trivial. Now if $\beta$ is nonsymmetric then it can not be all $1$'s, so there must be a near-concatenation of $\gamma$ with $\alpha$. However, $\gamma_1+\gamma_{\ell(\gamma)}=\alpha_1+\alpha_{\ell(\alpha)}\geq 3$, so is not a part of $\alpha$.\\

\item Chain \eqref{chain:4}: By Equation~\eqref{eq:circsize}, the total number of cells would have to factor in a way such that both factors admit nonsymmetric compositions, but neither $6=6\cdot 1$ nor $6=3\cdot 2$ satisfy this.\\

\item Chain \eqref{chain:5}: By Equation~\eqref{eq:circsize}, the total number of cells would have to factor in a way such that both factors admit nonsymmetric compositions, but $7=7\cdot 1$ does not satisfy this. \end{enumerate}\end{proof}

We now establish the order relations in the next two lemmas, from which Theorem \ref{thm:chains} will follow immediately.

\begin{lemma} \label{lem:celltrchains} The following relations hold.
\begin{enumerate}
\item All relations in Chain \eqref{chain:1a+1}
\item All relations in Chain \eqref{chain:1a}
\item Relation 1 in Chain \eqref{chain:4}
\item Relation 2 in Chain \eqref{chain:4}
\item Relation 1 in Chain \eqref{chain:5}
\item Relation 2 in Chain \eqref{chain:5}
\item Relation 3 in Chain \eqref{chain:5}
\end{enumerate}
\end{lemma}
\begin{proof} These follow from Theorem \ref{thm:celltr} with $a'=a+1$ and by setting $\delta$ and $\gamma$ as follows.
\begin{enumerate}
\item Use Case 1 with $\delta=a^{i-1}$ and $\gamma=a^{m-i}$ for $1\leq i\leq\lfloor\frac{m}{2}\rfloor$.
\item Use Case 1 with $\delta=(a+1)^{i-1}$ and $\gamma=(a+1)^{n-i}$ for $1\leq i\leq \lfloor\frac{n}{2}\rfloor$.
\item Use Case 1 with $\delta=\emptyset$ and $\gamma=(a+1)a$.
\item Use Case 1 with $\delta=(a+1)$ and $\gamma=a$.
\item Use Case 2 with $\delta=a(a+1)$ and $\gamma=a$.
\item Use Case 1 with $\delta=\emptyset$ and $\gamma=(a+1)aa$.
\item Use Case 1 with $\delta=a$ and $\gamma=a(a+1)$.
\end{enumerate}\end{proof}

There are now three remaining relations to establish. We will use a combination of Lemma \ref{lem:connearcon} and Theorem \ref{thm:moving1}, along with its generalization Theorem \ref{thm:moving1cor} to simplify  differences. Then we use the Littlewood-Richardson rule to calculate the Schur function expansions.

\begin{lemma} \label{lem:adversaries} The following relations hold.
\begin{enumerate}
\item Relation 3 in Chain \eqref{chain:4}
\item Relation 4 in Chain \eqref{chain:5}
\item Relation 5 in Chain \eqref{chain:5}
\end{enumerate}
\end{lemma}

\begin{proof} We first prove Parts 1 and 3 before proving Part 2, which is the most intricate.
\begin{enumerate}
 \item[(1)] We show that $r_{(a+1)(a+1)aa}\geq_s r_{(a+1)aa(a+1)}$. By Lemma \ref{lem:connearcon}, we have that \begin{align}\nonumber r_{(a+1)}r_{(a+1)aa}&=r_{(a+1)(a+1)aa}+r_{(2a+2)aa}\hbox{ and }\\\nonumber r_{(a+1)}r_{aa(a+1)}&=r_{(a+1)aa(a+1)}+r_{(2a+1)a(a+1)}.\end{align} Because the left sides are equal by Theorem~\ref{thm:hdl}, we now have $$r_{(a+1)(a+1)aa}-r_{(a+1)aa(a+1)}=r_{(2a+1)a(a+1)}-r_{(2a+2)aa}\geq_s 0$$by Theorem \ref{thm:moving1}.\\
 
\item[(3)] We show that $r_{(a+1)(a+1)aaa}\geq_s r_{(a+1)aaa(a+1)}$. Throughout this proof our diagrams illustrate the concrete case when $a=4$; however, the number of cells in each row in each diagram can be scaled appropriately to illustrate the generic case.
 
 By Lemma \ref{lem:connearcon}, we have that 
 \begin{align} \nonumber r_{(a+1)}r_{(a+1)aaa}&=r_{(a+1)(a+1)aaa}+r_{(2a+2)aaa}\hbox{ and }\\\nonumber r_{(a+1)}r_{aaa(a+1)}&=r_{(a+1)aaa(a+1)}+r_{(2a+1)aa(a+1)}.\end{align} Because the left sides are equal by Theorem~\ref{thm:hdl}, we have that
 \begin{align} \nonumber r_{(a+1)(a+1)aaa}-r_{(a+1)aaa(a+1)}&=r_{(2a+1)aa(a+1)}-r_{(2a+2)aaa}\\\nonumber& =\left(r_{(2a+1)a(a+1)a}-r_{(2a+2)aaa}\right)-\left(r_{(2a+1)a(a+1)a}-r_{(2a+1)aa(a+1)}\right).\end{align} 
 By Theorem \ref{thm:moving1}, 
 $$F_1=r_{(2a+1)a(a+1)a}-r_{(2a+2)aaa}=\sum_\nu c'_{(2a+1)a(a+1)(a+1),3,\nu}s_\nu$$is Schur-positive and is illustrated by the following diagram.\begin{equation}\label{diag:adv2,a}\tableau{&&&&&&&&&&1&1&1&1&1&1&1&1&1\\&&&&&&&\hspace{0pt}&\hspace{0pt}&\hspace{0pt}&2\\&&&\ast&\bullet&\hspace{0pt}&\hspace{0pt}&\hspace{0pt}\\\hspace{0pt}&\hspace{0pt}&\hspace{0pt}&\bullet}\end{equation}
 
 Similarly, by Theorem \ref{thm:moving1}, $$F_2=r_{(2a+1)a(a+1)a}-r_{(2a+1)aa(a+1)}=r_{a(a+1)a(2a+1)}-r_{(a+1)aa(2a+1)}$$ is Schur-positive and is illustrated by the following diagram. $$\tableau{&&&&&&&&&&&&&&&1&1&1&1\\
 &&&&&&&&&&&\ast&\bullet&\hspace{0pt}&\hspace{0pt}&2\\
 &&&&&&&&\hspace{0pt}&\hspace{0pt}&\hspace{0pt}&\bullet\\
\hspace{0pt}&\hspace{0pt}&\hspace{0pt}&\hspace{0pt}&\hspace{0pt}&\hspace{0pt}&\hspace{0pt}&\hspace{0pt}&\hspace{0pt}}$$\\
 
 Here if $\ast=2$ there would be too many $2$'s in the second row, so $\ast=1$ and the right bullet is at least the bullet below, so must both be $2$'s. Then the third row can not have any more $2$'s so must have $(a-1)$ $1$'s. The fourth row will then have some number of $1$'s, some number $0\leq x\leq a-1$ of $2$'s, and some number $0\leq y\leq a+1$ of $3$'s. We can not have $x=y=0$ otherwise there would be a $1$ below a $1$. So
 $$F_2=\sum_{(x,y)\in A}s_{(4a+1-x-y) (a+1+x)y},\hbox{ where}$$ $$A=\{(x,y):\hbox{ }0\leq x\leq a-1,\hbox{ }0\leq y\leq a+1,\hbox{ }(x,y)\neq(0,0)\}.$$
 
 It remains to find restricted LR tableaux of shape $(2a+1)a(a+1)a$, as illustrated in Diagram \eqref{diag:adv2,a}, with content $(4a+1-x-y)(a+1+x)y$ for each $(x,y)\in A$. \\

 If $y=0$, then since $1\leq x\leq a-1$, we can form the restricted LR tableaux as follows, with $x$ $2$'s  {in total} in the fourth row.

 $$\tableau{&&&&&&&&&&1&1&1
&1&1&1&1&1&1\\&&&&&&&1&1&1&2\\&&&1&2&2&2&2\\1&\hspace{0pt}&\hspace{0pt}&2}$$\\
 
 If $x=0$ and $y=1$, then form the restricted LR tableau as follows.
 
  $$\tableau{&&&&&&&&&&1&1&1&1&1&1&1&1&1\\&&&&&&&1&1&1&2\\&&&1&2&2&2&3\\1&1&1&2}$$\\
  
   If $x=0$ and $2\leq y\leq a+1$, then form the restricted LR tableaux as follows, with $(y-1)$ $3$'s  {in total} preceded by $1$'s in the fourth row.
 
  $$\tableau{&&&&&&&&&&1&1&1&1&1&1&1&1&1\\&&&&&&&1&1&1&2\\&&&2&2&2&2&3\\\hspace{0pt}&\hspace{0pt}&\hspace{0pt}&3}$$\\
  
  If $x\neq 0$ and $y=a+1$, then form the restricted LR tableaux as follows, with $x$ additional $2$'s preceded by $1$'s in the second row.
  
  $$\tableau{&&&&&&&&&&1&1&1&1&1&1&1&1&1\\&&&&&&&\hspace{0pt}&\hspace{0pt}&\hspace{0pt}&2\\&&&2&2&2&2&3\\3&3&3&3}$$\\
  
  Otherwise, $x\neq 0$ and $1\leq y\leq a$, and we form the restricted LR tableaux as follows, with $x$ $2$'s  {in total} preceded by $1$'s in the second row, and $y$ $3$'s in total preceded by $1$'s, in the fourth row.
  
  $$\tableau{&&&&&&&&&&1&1&1&1&1&1&1&1&1\\&&&&&&&1&\hspace{0pt}&\hspace{0pt}&2\\&&&2&2&2&2&2\\\hspace{0pt}&\hspace{0pt}&\hspace{0pt}&3}$$\
 
 This completes the proof of this part.\\

 \item[(2)] We show that $r_{(a+1)aa(a+1)a}\geq_s r_{(a+1)(a+1)aaa}$. Throughout this proof our diagrams illustrate the concrete case when $a=4$; however, the number of cells in each row in each diagram can be scaled appropriately to illustrate the generic case.
 
 By Lemma \ref{lem:connearcon}, we have that 
 \begin{align}\nonumber r_{(a+1)}r_{aa(a+1)}r_a&=r_{(a+1)aa(a+1)a}+r_{(a+1)aa(2a+1)}+r_{(2a+1)a(a+1)a}+r_{(2a+1)a(2a+1)}\hbox{ and }\\\nonumber r_{(a+1)}r_{(a+1)aa}r_a&=r_{(a+1)(a+1)aaa}+r_{(a+1)(a+1)a(2a)}+r_{(2a+2)aaa}+r_{(2a+2)a(2a)}.\end{align} 
 Because the left sides are equal by Theorem~\ref{thm:hdl}, we have that \begin{align}\nonumber &r_{(a+1)aa(a+1)a}-r_{(a+1)(a+1)aaa}\\\nonumber&=\left(r_{(a+1)(a+1)a(2a)}-r_{(a+1)aa(2a+1)}\right)-\left(r_{(2a+1)a(a+1)a}-r_{(2a+2)aaa}\right)\\\nonumber&-\left(r_{(2a+1)a(2a+1)}-r_{(2a+2)a(2a)}\right)\\\nonumber&=\left((r_{(2a)(a+1)(a+1)a}-r_{(2a+1)(a+1)aa})-(r_{(2a+1)a(a+1)a}-r_{(2a+2)aaa})\right)\\\nonumber&-\left((r_{(2a)(a+1)(a+1)a}-r_{(2a+1)(a+1)aa})-(r_{(2a)a(a+1)(a+1)}-r_{(2a+1)aa(a+1)})\right)\\\nonumber&-\left((r_{(2a+1)(a+1)(2a)}-r_{(2a+2)a(2a)})-(r_{(2a)(a+1)(2a+1)}-r_{(2a+1)a(2a+1)})\right).\end{align}

We will now use Theorems \ref{thm:moving1} and \ref{thm:moving1cor} to analyze the three differences \begin{align}\nonumber G_1&=\left((r_{(2a)(a+1)(a+1)a}-r_{(2a+1)(a+1)aa})-(r_{(2a+1)a(a+1)a}-r_{(2a+2)aaa})\right),\\\nonumber G_2&=\left((r_{(2a)(a+1)(a+1)a}-r_{(2a+1)(a+1)aa})-(r_{(2a)a(a+1)(a+1)}-r_{(2a+1)aa(a+1)})\right),\hbox{ and }\\\nonumber G_3&=\left((r_{(2a+1)(a+1)(2a)}-r_{(2a+2)a(2a)})-(r_{(2a)(a+1)(2a+1)}-r_{(2a+1)a(2a+1)})\right),\end{align}
 in order to show that $$r_{(a+1)aa(a+1)a}-r_{(a+1)(a+1)aaa}=G_1-G_2-G_3$$ is Schur-positive.\\

 By Theorem \ref{thm:moving1cor}, the sum \begin{align*}G_1&=\left((r_{(2a)(a+1)(a+1)a}-r_{(2a+1)(a+1)aa})-(r_{(2a+1)a(a+1)a}-r_{(2a+2)aaa})\right)\\&=\sum_\nu c'_{(2a)(a+1)(a+1)a,2,3,\nu}s_\nu\end{align*} is Schur-positive and is illustrated by the diagram below. If an asterisk is a $1$, then the bullet to its right must be at least the bullet below that asterisk.
 
 \begin{equation}\label{diag:adv2}\tableau{&&&&&&&&&&&1&1&1&1&1&1&1&1\\&&&&&&&\ast&\bullet&\hspace{0pt}&\hspace{0pt}&2\\&&&\ast&\bullet&\hspace{0pt}&\hspace{0pt}&\bullet\\\hspace{0pt}&\hspace{0pt}&\hspace{0pt}&\bullet}\end{equation}
 
 Now consider \begin{align}\nonumber G_2&=\left((r_{(2a)(a+1)(a+1)a}-r_{(2a+1)(a+1)aa})-(r_{(2a)a(a+1)(a+1)}-r_{(2a+1)aa(a+1)})\right)\\\nonumber&=\left((r_{(2a)(a+1)(a+1)a}-r_{(2a)a(a+1)(a+1)})-(r_{(2a+1)(a+1)aa}-r_{(2a+1)aa(a+1)})\right).\end{align} By Lemma \ref{lem:connearcon}, we have \begin{align}\nonumber r_{(2a)}r_{(a+1)(a+1)a}&=r_{(2a)(a+1)(a+1)a}+r_{(3a+1)(a+1)a}\hbox{ and }\\\nonumber r_{(2a)}r_{a(a+1)(a+1)}&=r_{(2a)a(a+1)(a+1)}+r_{(3a)(a+1)(a+1)}.\end{align} Because the left sides are equal by Theorem~\ref{thm:hdl}, we find that $$r_{(2a)(a+1)(a+1)a}-r_{(2a)a(a+1)(a+1)}=r_{(3a)(a+1)(a+1)}-r_{(3a+1)(a+1)a}.$$
 Similarly, by Lemma \ref{lem:connearcon} we have \begin{align}\nonumber r_{(2a+1)}r_{(a+1)aa}&=r_{(2a+1)(a+1)aa}+r_{(3a+2)aa}\hbox{ and }\\\nonumber r_{(2a+1)}r_{aa(a+1)}&=r_{(2a+1)aa(a+1)}+r_{(3a+1)a(a+1)}.\end{align} Because the left sides are equal by Theorem~\ref{thm:hdl}, we find that $$r_{(2a+1)(a+1)aa}-r_{(2a+1)aa(a+1)}=r_{(3a+1)a(a+1)}-r_{(3a+2)aa}.$$
 Therefore, we have that 
 $$G_2=\left((r_{(3a)(a+1)(a+1)}-r_{(3a+1)(a+1)a})-(r_{(3a+1)a(a+1)}-r_{(3a+2)aa})\right);$$ by Theorem \ref{thm:moving1cor}, $G_2$ is Schur-positive and is illustrated by the following diagram.
 
 $$\tableau{&&&&&&&&1&1&1&1&1&1&1&1&1&1&1&1\\&&&&\ast&\bullet&\hspace{0pt}&\hspace{0pt}&2\\\ast&\hspace{0pt}&\hspace{0pt}&\hspace{0pt}&\bullet}$$\\
 
 If the top asterisk is a 1, then both bullets must be $2$'s and the third row is all $2$'s. If the top asterisk is a 2, then the second row is all 2's and the third row is some number $0\leq x\leq a$ of $2$'s followed by $(a+1-x)$ $3$'s. Therefore we have that
 $$G_2=s_{({3a+1})({2a+1})}+\sum_{x=0}^as_{({3a})({a+1+x})({a+1-x})}.$$
 
 Now let us consider $$G_3=\left((r_{(2a+1)(a+1)(2a)}-r_{(2a+2)a(2a)})-(r_{(2a)(a+1)(2a+1)}-r_{(2a+1)a(2a+1)})\right).$$
 By Theorem \ref{thm:moving1}, the first difference is Schur-positive and is illustrated by the following diagram.
 
 $$\tableau{&&&&&&&&&&&1&1&1&1&1&1&1&1&1\\&&&&&&&\ast&\bullet&\hspace{0pt}&\hspace{0pt}&2\\\hspace{0pt}&\hspace{0pt}&\hspace{0pt}&\hspace{0pt}&\hspace{0pt}&\hspace{0pt}&\hspace{0pt}&\bullet}$$\\
 
 If $\ast=1$ then both bullets are $2$'s and third row is some number $0\leq x\leq a+1$ additional $2$'s preceded by $1$'s (or $0\leq x\leq a$ if $a=1$). If $\ast=2$ then the second row is all $2$'s and the third row is some number of $1$'s followed by some number $0\leq x\leq a$ of $2$'s, some number $0\leq y\leq a$ of $3$'s, and one more 3 for the bottom bullet. So we have
 \begin{equation}\label{eq:adv2,c,1}r_{(2a+1)(a+1)(2a)}-r_{(2a+2)a(2a)}=s_{(3a)(2a+2)}+\sum_{x=0}^as_{(4a+1-x)(a+1+x)}+\sum_{0\leq x,y \leq a \atop 0\leq x+y\leq 2a-1}s_{(4a-x-y)(a+1+x)(1+y)},\end{equation} where by convention if $a=1$, then $s_{(3a)(2a+2)}=s_{34}=0$.\\
 
 Similarly, by Theorem \ref{thm:moving1}, the second difference is Schur-positive and is illustrated by the following diagram.
 
 $$\tableau{&&&&&&&&&&&&1&1&1&1&1&1&1&1\\&&&&&&&&\ast&\bullet&\hspace{0pt}&\hspace{0pt}&2\\\hspace{0pt}&\hspace{0pt}&\hspace{0pt}&\hspace{0pt}&\hspace{0pt}&\hspace{0pt}&\hspace{0pt}&\hspace{0pt}&\bullet}$$\\
  
  If $\ast=1$ then both bullets are $2$'s and the third row is some number $0\leq x\leq a$ additional 2's preceded by $(2a-x)$ $1$'s. If $\ast=2$ then the second row is all $2$'s and the third row is some number of $1$'s followed by some number $0\leq x\leq a-1$ of $2$'s, some number $0\leq y\leq a$ of $3$'s, and one more 3 for the bottom bullet. So we have
  \begin{equation}\label{eq:adv2,c,2}r_{(2a)(a+1)(2a+1)}-r_{(2a+1)a(2a+1)}=\sum_{x=0}^as_{({4a+1-x})({a+1+x})}+\sum_{y=0}^a\sum_{x=0}^{a-1}s_{({4a-x-y})({a+1+x})({1+y})}.\end{equation}
  
  Subtracting Equation \eqref{eq:adv2,c,2} from Equation \eqref{eq:adv2,c,1}, we find that  $$G_3=s_{({3a})({2a+2})}+\sum_{y=0}^{a-1}s_{({3a-y})({2a+1})({1+y})}.$$
 
 Therefore, we must show that $$\sum_\nu c'_{(2a)(a+1)(a+1)a,2,3,\nu}s_\nu=G_1\geq_s G_2+G_3=\sum_{(x,y)\in B}s_{(4a+1-x-y)(a+1+x)y},$$ where $B$ is the multiset union $$B=\{(x,a+1-x):\hbox{ }0\leq x\leq a+1\}\cup\{(a,y):\hbox{ }0\leq y\leq a\}.$$ It remains to find restricted LR tableaux of shape $(2a)(a+1)(a+1)a$, as illustrated in Diagram \eqref{diag:adv2}, with content $(4a+1-x-y)(a+1+x)y$ for each $(x,y)\in B$. We must find two fillings where $(x,y)=(a,1)$, and we can ignore the pair $(x,y)=(a+1,0)$ if $a=1$.\\
 
 If $x=a+1$, $y=0$, and $a\geq 2$, then form the restricted LR tableau as follows.
 
 $$\tableau{&&&&&&&&&&&1&1&1&1&1&1&1&1\\&&&&&&&1&2&2&2&2\\&&&1&2&2&2&2\\1&1&2&2}$$\\
 
  If $x=a$ and $y=0$ then form the restricted LR tableau as follows.
 
 $$\tableau{&&&&&&&&&&&1&1&1&1&1&1&1&1\\&&&&&&&1&2&2&2&2\\&&&1&2&2&2&2\\1&1&1&2}$$\\
 
 If $x=a$ and $y=1$ then here are two restricted LR tableaux.
 
 $$\tableau{&&&&&&&&&&&1&1&1&1&1&1&1&1\\&&&&&&&2&2&2&2&2\\&&&1&2&2&2&3\\1&1&1&2}$$
 
 $$\tableau{&&&&&&&&&&&1&1&1&1&1&1&1&1\\&&&&&&&1&2&2&2&2\\&&&2&2&2&2&2\\1&1&1&3}$$\\
 
 If $x=a$ and $2\leq y \leq a$ then form the restricted LR tableaux as follows, with $y$ $3$'s in total preceded by $1$'s in the fourth row.
 
 $$\tableau{&&&&&&&&&&&1&1&1&1&1&1&1&1\\&&&&&&&1&2&2&2&2\\&&&2&2&2&2&2\\1&\hspace{0pt}&\hspace{0pt}&3}$$\\
 
 If $x=0$ and $y=a+1$ then form the restricted LR tableau as follows.
 
 $$\tableau{&&&&&&&&&&&1&1&1&1&1&1&1&1\\&&&&&&&2&2&2&2&2\\&&&1&3&3&3&3\\1&1&1&3}$$\\
 
 Finally, if $x+y=a+1$ with $1\leq x\leq a-1$ then form the restricted LR tableaux as follows, with $(x-1)$ $2$'s and $y$ $3$'s in total in the third row.
 
 $$\tableau{&&&&&&&&&&&1&1&1&1&1&1&1&1\\&&&&&&&2&2&2&2&2\\&&&1&\hspace{0pt}&\hspace{0pt}&\hspace{0pt}&3\\1&1&1&2}$$\\This completes the proof of this part, and therefore, the proof of this lemma.\end{enumerate}\end{proof}
\section{Large ribbons and short end rows}\label{sec:shortends}
In the previous section, Chain \eqref{chain:1a} from Theorem \ref{thm:chains} suggests that a large ribbon in $\mathcal{R}((a+1)^na^m)$ tends to have short end rows. This is also corroborated by Theorem \ref{thm:moving1}, which can produce a larger ribbon in $\mathcal{P}_N$ by shortening the first row by one cell. Additionally, Lemma \ref{lem:eqrib} tells us that the conjectured maximal element of $\mathcal{P}_N$ has a short last row. In this section, we prove the following necessary condition for an order relation in $\mathcal{R}((a+1)^na^m)$ along these lines.

\begin{theorem} \label{thm:shortends} Let $\alpha,\beta\in \mathcal{R}((a+1)^na^m)$. If $\ra\geq_s\rb$, then $SE(\alpha)\geq SE(\beta)$.  \end{theorem}

\begin{example} \label{ex:shortendschain} Chain \eqref{chain:4} with $a=4$ gives the following. $$r_{4554}>_s r_{5454}>_s r_{5544}>_s r_{5445}$$ The numbers of short ends of the ribbons are respectively $2$, $1$, $1$, and $0$, weakly decreasing as the ribbon Schur functions decrease.\end{example}

\begin{remark} \label{rem:convexsubse} It follows immediately from Theorem \ref{thm:shortends} that the class of ribbons in $\mathcal{R}((a+1)^na^m)$ with zero (similarly one or two) short ends forms a convex subposet of $\mathcal{R}((a+1)^na^m)$. \end{remark}

We now prove Theorem \ref{thm:shortends}. Because the cases where $a=1$ and $a\geq 2$ are qualitatively different, we handle them separately.

\begin{lemma} \label{lem:shortendsa=1} Theorem \ref{thm:shortends} holds if $a=1$. \end{lemma}

\begin{proof} We may assume that $m,n\geq 1$, otherwise $\mathcal{R}(2^n1^m)$ only consists of one element so $\alpha=\beta$. Suppose that $SE(\alpha)<SE(\beta)$. We will identify a positive term of $\rb-\ra$, contradicting $\ra\geq_s\rb$.\\

Let $\delta\in \mathcal{R}(2^n1^m)$ and let $\sigma$ be the content of the lexicographically largest LR tableau filling of shape $\delta$. We then have \begin{align}\nonumber \sigma_1&=\hbox{number of $1$'s in largest filling}\\\nonumber&=\hbox{number of columns of $\delta$}\\\nonumber&=\hbox{number of cells of $\delta$}-\hbox{number of rows of $\delta$}+1\\\nonumber&=(2n+m)-(n+m)+1=n+1.\end{align} We also calculate \begin{align}\nonumber \sigma_2&=\hbox{number of $2$'s in largest filling}\\\nonumber&=\hbox{number of columns of $\delta$ of length at least $2$}\\\nonumber&=\hbox{number of columns of $\delta$}-\hbox{number of single-cell columns of $\delta$}.\end{align} Because by assumption $\delta$ has at least two rows and all rows are of length $1$ or $2$, we can verify that a single-cell column occurs exactly when $\delta$ has a $2$ as its first part or as its last part. Therefore $$\sigma_2=(n+1)-(2-SE(\delta))=n-1+SE(\delta).$$

Now let $\nu$ be the content of the lexicographically largest LR tableau filling of shape $\beta$ and let $\rho$ be the content of that of $\alpha$. By the above, we know that $\nu_1=n+1=\rho_1$, and because $SE(\alpha)<SE(\beta)$, that $\rho_2<\nu_2$. Therefore $\rho<_{lex}\nu$, and the lexicographically largest term of $\rb-\ra$ is $s_\nu$, a positive term. This contradicts $\ra\geq_s\rb$.\end{proof}

\begin{example} \label{ex:shortendsa=1} The lexicographically largest LR tableau fillings of shapes $21212$, $12212$, and $12221$, given below, have contents $422$, $431$, and $44$ respectively. The number of $1$'s in each case is $3+1=4$ and the number of $2$'s is $2$ plus the number of short ends.
$$\tableau{&&1&1\\&&2\\&1&3\\&2\\1&3}\tableau{&&&1\\&&1&2\\&1&2\\&2\\1&3}\tableau{&&&1\\&&1&2\\&1&2\\1&2\\2}$$ For example, the difference $$r_{12221}-r_{12212}=(s_{44}+\hbox{(smaller terms)})-(s_{431}+\hbox{(smaller terms)})=s_{44}\pm\hbox{(smaller terms)}$$ contains the positive term $s_{44}$, so $r_{12212}\ngeq_s r_{12221}$.\end{example}

Before proving the $a\geq 2$ case of Theorem \ref{thm:shortends}, we make a quick definition.

\begin{definition} \label{def:comfortable} An LR tableau $T$ of equitable ribbon shape is \emph{comfortable} if every long row has a $1$ and every long intermediate row has two $1$'s. Otherwise, $T$ is \emph{uncomfortable}. \end{definition}

\begin{example} \label{ex:comfortable} Below are two LR tableaux of shape $5544$. The left tableau is comfortable; note that the last row is short so it need not have a $1$. The right tableau is uncomfortable because the second row, which is long, has only one $1$.

$$\tableau{&&&&&&&&&&1&1&1&1&1\\&&&&&&1&1&2&2&2\\&&&1&2&2&2\\1&1&2&2\\&}\hspace{0pt}\tableau{&&&&&&&&&&1&1&1&1&1\\&&&&&&1&2&2&2&2\\&&&1&1&2&2\\1&1&2&2\\&}$$ \end{example}

\begin{lemma} \label{lem:shortendsageq2} Theorem \ref{thm:shortends} holds if $a\geq 2$. \end{lemma}

\begin{proof} We may assume that $m,n\geq 1$, otherwise $\mathcal{R}((a+1)^na^m)$ only consists of one element so $\alpha=\beta$. Suppose that $SE(\alpha)<SE(\beta)$. We will identify a positive term of $\rb-\ra$, contradicting $\ra\geq_s\rb$.\\

Let $\nu$ be the two-part partition $(aR+n-(a+R-2))(a+R-2)$; we will count LR tableaux fillings of shapes $\alpha$ and $\beta$ of content $\nu$. First let us describe such fillings.\\

The lattice word condition tells us that the first row must be all $1$'s and that the $(a+R-2)$ $2$'s must each be preceded by their own $1$. Because we only have $1$'s and $2$'s, each of the $(R-1)$ two-cell columns must be filled with a $1$ above a $2$. Then there are $(a-1)$ remaining $2$'s left to place. Now because the first row has $a$ or $(a+1)$ $1$'s, only one of which was in a two-cell column, there are enough $1$'s in the first row to satisfy the lattice word condition. In other words, any SSYT of shape $\alpha$ or $\beta$ with content $\nu$ and with all $1$'s in the top row is automatically an LR tableau.\\

 It is useful to count separately the comfortable and uncomfortable LR tableaux. Let $A_c$ and $A_u$ be respectively the sets of comfortable and uncomfortable LR tableaux of shape $\alpha$ and content $\nu$, and let $B_c$ and $B_u$ be respectively the sets of comfortable and uncomfortable LR tableaux of shape $\beta$ and content $\nu$.\\ 

First we handle the comfortable tableaux by defining a bijection $$g:A_c\rightarrow B_c.$$ Given an LR tableau $T\in A_c$, set $g(T)$ to be the tableau of shape $\beta$ where the $k$-th row is filled as follows:\begin{itemize}
\item If $\alpha_k=\beta_k$, fill the $k$-th row of $g(T)$ in the same way as the $k$-th row of $T$.
\item If $\alpha_k=a$ and $\beta_k=a+1$, fill the $k$-th row of $g(T)$ with a $1$ followed by the entries of the $k$-th row of $T$.
\item If $\alpha_k=a+1$ and $\beta_k=a$, fill the $k$-th row of $g(T)$ with the second through $(a+1)$-th entries of the $k$-th row of $T$.
\end{itemize}
Because $T$ is comfortable, all long rows of $T$ have a $1$, so the content is preserved. Also, because all long intermediate rows of $T$ have two $1$'s, the resulting tableau $g(T)$ is still an SSYT. By construction, $g(T)$ has all $1$'s in the top row, is of shape $\beta$, and is comfortable, so indeed $g(T)\in B_c$. Finally, $g$ is a bijection because the inverse can be constructed in exactly the same way. Therefore, $|A_c|=|B_c|$.\\

We now handle the uncomfortable tableaux. Let us characterize $T\in A_u$. We know that the first row of $T$ is all $1$'s, the $(R-1)$ two-cell columns of $T$ have a $1$ followed by a $2$, and that as long as $T$ is an SSYT, the lattice word condition will be satisfied, no matter how the remaining $(a-1)$ $2$'s are distributed among the remaining rows. The bottom row of $T$, if it happens to be long, will have $a$ remaining cells, so it must contain a $1$. Therefore, for $T$ to be uncomfortable, there must be a long intermediate row $i$ that fails to have two $1$'s. Since row $i$ has $(a+1)$ cells, the first of which is a $1$ and the last of which is a $2$, the $(a-1)$ remaining cells must be filled with the $(a-1)$ remaining $2$'s, thus determining the entire LR tableau. Hence $|A_u|$ is the number of long intermediate rows of $\alpha$; similarly $|B_u|$ is the number of long intermediate rows of $\beta$. Therefore $|B_u|-|A_u|=SE(\beta)-SE(\alpha)$.\\

Now the difference $\rb-\ra$ contains the positive term $$((|B_c|+|B_u|)-(|A_c|+|A_u|))s_\nu=(|B_u|-|A_u|)s_\nu=(SE(\beta)-SE(\alpha))s_\nu,$$ contradicting $\ra\geq_s\rb$. \end{proof}

\begin{example} \label{ex:uncomfortable} Below are the two uncomfortable tableaux of shape $5554$ and content $13\hbox{ }6$. For each long intermediate row of the ribbon, there is a unique such tableau where that row is filled with a single $1$ followed by $2$'s.

\begin{align*}&\tableau{&&&&&&&&&&&1&1&1&1&1\\&&&&&&&1&2&2&2&2\\&&&1&1&1&1&2\\1&1&1&2\\&}\\&\tableau{&&&&&&&&&&&1&1&1&1&1\\&&&&&&&1&1&1&1&2\\&&&1&2&2&2&2\\1&1&1&2\\&}\end{align*} \end{example}

\begin{remark} \label{rem:shortendsterms} The partitions used in the proofs of Lemmas \ref{lem:shortendsa=1} and \ref{lem:shortendsageq2} are the lexicographically largest partitions appearing in the difference $\rb-\ra$. There are examples where these are the only terms occurring, the smallest of which are $$r_{121}-r_{112}=s_{22}\hbox{ and }r_{232}-r_{223}=s_{43},$$ so in that sense this result is the best possible. \end{remark}

It is tempting to assert the converse statement that for $\alpha,\beta\in \mathcal{R}((a+1)^na^m)$, if $SE(\alpha)>SE(\beta)$, then $\ra\geq_s\rb$ as well; however, we will see counterexamples in Corollary \ref{cor:incompar} of Section \ref{sec:coarsenings}. We can however prove the following partial converse, which resembles \cite[Proposition 3.5]{maxsupport}.

\begin{proposition} \label{prop:shortendsconv} Let $\beta\in \mathcal{R}((a+1)^na^m)$ with $m\geq 2$. Then there is an $\alpha\in \mathcal{R}((a+1)^na^m)$ such that $SE(\alpha)=2$ and $$\ra\geq_s\rb.$$ \end{proposition}

\begin{remark} \label{rem:shortendsconv=1} If $m=1$, then $\mathcal{R}((a+1)^na)$ is Chain \eqref{chain:1a} and by Theorem \ref{thm:chains} the full converse to Theorem \ref{thm:shortends} does in fact hold. This is why we assume that $m\geq 2$ in Proposition \ref{prop:shortendsconv}. \end{remark}

\begin{proof} If $SE(\beta)=2$ we simply take $\alpha=\beta$. Otherwise, reversing if necessary, we can write $\beta$ in the form $$\beta=\delta(a+1)a\gamma,$$ where $\delta$ is a possibly empty string of all $(a+1)$'s and because $m\geq 2$, $\gamma$ has a part of size $a$.\\

 If $\ell(\gamma)\geq\ell(\delta)$, then because $\delta^*$ is all $(a+1)$'s it dominates $\gamma$ up to $\ell(\delta)$. If $\ell(\gamma)<\ell(\delta)$, then because $\gamma$ has an $a$, then $\delta^*$ dominates $\gamma$ up to part $(\ell(\gamma)-1)$ and $\delta^*$ strictly dominates $\gamma$ at $\ell(\gamma)$. In either case, by Theorem \ref{thm:celltr}, we have that $$r_{\delta a(a+1)\gamma}\geq_sr_{\delta(a+1)a\gamma}.$$

This allows us to iteratively find successively larger ribbons until we obtain a ribbon with short first row. If necessary, we reverse and continue iterating to obtain a still larger ribbon $\alpha$ with two short ends.\end{proof}

\begin{example} \label{ex:se}Given $\beta=555454455$, our procedure  finds the following chain, culminating in $\alpha=455455554$. \begin{align*}\rb&=r_{555454455}<_s r_{554554455}<_s r_{545554455}<_s r_{455554455}\\&=r_{554455554}<_s r_{545455554}<_s r_{455455554}=\ra.\end{align*}\end{example}

\begin{corollary} \label{cor:shortends} The ribbon $\alpha=a(a+1)^na$ is the unique maximal element of $\mathcal{R}((a+1)^na^2)$. \end{corollary} 

\begin{proof} This ribbon $\alpha$ is the unique element of $\mathcal{R}((a+1)^na^2)$ with $SE(\alpha)=2$. Therefore, by Proposition \ref{prop:shortendsconv}, we have $\ra\geq_s\rb$ for every $\beta\in \mathcal{R}((a+1)^na^2)$.\end{proof}

\begin{remark} \label{rem:semaxforposets} According to \eqref{eq:shortendscor} of Corollary \ref{cor:maxforposets}, Corollary \ref{cor:shortends} confirms that for $\mathcal{R}((a+1)^na^2)$ the box diagonal diagram is indeed the unique maximal element, confirming Conjecture~\ref{conj:maxel} in this case. \end{remark} 

\begin{remark} \label{rem:sedirections} It is tempting to try to extend Theorem \ref{thm:shortends} to the rows next to the end rows and continuing to work inward. For example, we may think that if $\ra\geq_s\rb$ and that $SE(\alpha)=SE(\beta)$, it would be necessary that the number of short rows of $\alpha$ next to the end rows must exceed that of $\beta$. However, the inequality $$r_{a(a+1)a(a+1)a}\geq_sr_{a(a+1)(a+1)aa}$$ from Chain \eqref{chain:5} shows that this is not the case. In Section \ref{sec:coarsenings} we explain what is going on here. \end{remark}

\section{Large ribbons and even distribution of long rows}\label{sec:coarsenings}

In this section we present our second necessary condition for an order relation in $\mathcal{R}((a+1)^na^m)$. We begin with an example to motivate our result. \\

According to \eqref{eq:smallscor} of Corollary \ref{cor:maxforposets}, the conjectured maximal element of Conjecture~\ref{conj:maxel} for $\mathcal{R}(5^44^{15})$ is $$\beta=444\hbox{ }5\hbox{ }444\hbox{ }5\hbox{ }444\hbox{ }5\hbox{ }444\hbox{ }5\hbox{ }444,$$ suggesting that the long rows, namely the $5$'s, are evenly distributed, far from each other and far from the ends of the ribbon. Our goal for this section is to show that indeed for $\alpha,\beta\in \mathcal{R}((a+1)^na^m)$ to satisfy $\ra\geq_s\rb$, the long rows of $\alpha$ must be more evenly distributed than those of $\beta$. We now make this notion precise.

\begin{definition} \label{def:profiles} Let us write a composition $\alpha\in \mathcal{R}((a+1)^na^m)$ as $$\alpha=a^{p_1}(a+1)a^{p_2}(a+1)\cdots(a+1)a^{p_{n+1}},$$ where $p_i\geq 0$ for $1\leq i\leq n+1$. Then we define the \emph{profile} of $\alpha$ to be the tuple $$p(\alpha)=p_1p_2\cdots p_{n+1}$$ and we define the \emph{quasi-profile} of $\alpha$ to be the tuple $q(\alpha)=q_0q_1\cdots$, where $$q_j=|\{i:\hbox{ }p_i=j\}|;$$ that is, $q_j$ is the number of $j$'s in $p(\alpha)$.\end{definition}

\begin{remark} \label{rem:profiles0} The way it is defined, the quasi-profile has an inifinite tail of $0$'s. We will omit this for brevity. Let us also note that $\sum_{i=1}^{n+1}p_i=m$.\end{remark}

\begin{example} \label{ex:profiles} Let $$\alpha=44444444\hbox{ }5\hbox{ }44444\hbox{ }5\hbox{ }\hbox{ }\hbox{ }5\hbox{ }44\hbox{ }5.$$ Then $$p(\alpha)=85020\hbox{ and }q(\alpha)=201001001.$$ \end{example}

\begin{remark} \label{rem:profiles} Given a fixed value of $a$, one can recover a composition $\alpha$ from its profile $p(\alpha)$. However, one can not in general recover $\alpha$ from its quasi-profile $q(\alpha)$ because compositions whose profiles are permuted will have the same quasi-profile.  \end{remark}

We are now able to state our result regarding the distribution of long rows. The proof of this result will require much preparation, namely Corollary \ref{cor:observation}, Theorem \ref{thm:jtcoarse}, and Theorem \ref{thm:ftom}.

\begin{theorem} \label{thm:smalls} Let $\alpha,\beta\in \mathcal{R}((a+1)^na^m)$. If $\ra\geq_s\rb$, then $q(\alpha)\leq_{lex}q(\beta)$. \end{theorem}

We will first look at some examples and make some remarks.

\begin{example} \label{ex:smalls} Let $$\alpha=44444444\hbox{ }5\hbox{ }44444\hbox{ }5\hbox{ }\hbox{ }\hbox{ }5\hbox{ }44\hbox{ }5\hbox{ and }\beta=444\hbox{ }5\hbox{ }444\hbox{ }5\hbox{ }444\hbox{ }5\hbox{ }444\hbox{ }5\hbox{ }444.$$ Then $$p(\alpha)=85020,\hbox{ }q(\alpha)=201001001,\hbox{ }p(\beta)=33333,\hbox{ and }q(\beta)=0005.$$ Because $q(\alpha)>_{lex}q(\beta)$, we know by Theorem \ref{thm:smalls} that $\ra\ngeq_s\rb$.\end{example}

\begin{remark} \label{rem:smalls} We observe that the quasi-profile of an $\alpha\in \mathcal{R}((a+1)^na^m)$ represents how evenly distributed its long rows are. Now Theorem \ref{thm:smalls} allows us to understand the inequality $$r_{a(a+1)a(a+1)a}\geq_sr_{a(a+1)(a+1)aa}$$ presented in Remark \ref{rem:sedirections} of the previous section. After maximizing the number of short ends, the next priority in looking for a large ribbon in $\mathcal{R}((a+1)^na^m)$ is not to place more short rows toward the ends of the ribbon, but rather to evenly distribute the long rows. \end{remark}

\begin{example} \label{ex:smallschain} Chain \eqref{chain:4} with $a=4$ gives the following. $$r_{4554}>_sr_{5454}>_sr_{5544}>_sr_{5445}$$
 The quasi-profiles of the ribbons are respectively $12$, $12$, $201$, and $201$, weakly increasing lexicographically as the ribbon Schur functions decrease.\end{example}

\begin{remark} \label{rem:smallschain} From Examples \ref{ex:shortendschain} and \ref{ex:smallschain}, we see that if one is initially only told that $\mathcal{R}(5^24^2)$ is a chain, then Theorems \ref{thm:shortends} and \ref{thm:smalls} are sufficient to put the ribbons in order. \end{remark}

\begin{remark} \label{rem:convexsubq} Similar to Remark \ref{rem:convexsubse} from Section \ref{sec:shortends}, it immediately follows from Theorem \ref{thm:smalls} that the parameter $q$ with lexicographic order partitions $\mathcal{R}((a+1)^na^m)$ into convex subposets according to its possible values. \end{remark} 

We will now prepare to prove Theorem \ref{thm:smalls}. We first use the Jacobi-Trudi  identity to obtain leading terms in the Schur function expansion of a ribbon Schur function. We then recast this identity in a combinatorial light and perform an extensive enumeration.\\

Recall that the Jacobi-Trudi identity states that for partitions $\lambda$ and $\mu$, we have \begin{equation} \label{eq:jt} s_{\lambda/\mu}=\det(h_{\lambda_i-\mu_j-i+j}),\end{equation} where $h_N=\sum _{i_1\leq \cdots \leq i_N} x_{i_1}\cdots x_{i_N}$ is the $N$-th complete homogeneous symmetric function. By convention, $h_{\lambda _1\cdots \lambda _{\ell(\lambda)}} = h_{\lambda _1} \cdots h_{\lambda _{\ell(\lambda)}}$, $h_0=1$ and $h_N=0$ if $N<0$.\\

\begin{lemma} \label{lem:observation} 
\hspace{2pt}
\begin{enumerate}
\item The lexicographically least term of the complete homogeneous symmetric function expansion of $s_\lambda$ is $h_\lambda$. In other words, we can write $$s_\lambda=h_\lambda+\sum_{\mu>_{lex}\lambda}c_\mu h_\mu$$ for some coefficients $c_\mu$. \\
\item The lexicographically least term of the Schur function expansion of $h_\lambda$ is $s_\lambda$. In other words, we can write $$h_\lambda=s_\lambda+\sum_{\mu>_{lex}\lambda}c'_\mu s_\mu$$ for some coefficients $c'_\mu$. \end{enumerate}\end{lemma}

\begin{proof}  \begin{enumerate} \item We use induction on the number of parts of $\lambda$. If $\lambda$ has zero parts, then $s_\emptyset=1=h_\emptyset$. Otherwise, expanding the determinant of Equation \eqref{eq:jt} along the first row, we have that for $\lambda = \lambda _1\lambda_2\cdots \lambda _{\ell(\lambda)}$ $$s_\lambda=h_{\lambda_1}\det(h_{\lambda_{i+1}-i+j})+\sum_{j'=2}^{\ell(\lambda)} h_{\lambda_1-1+j'}\det(A_{j'}),$$
where $A_{j'}$ is the minor obtained by removing the first row and $j'$-th column from the original matrix. The partitions appearing in the sum have a part $\lambda_1-1+j'>\lambda_1$ so are lexicographically larger than $\lambda$, while by the induction hypothesis the lexicographically least term of $\det(h_{\lambda_{i+1}-i+j})$ is $h_{\lambda_2\cdots\lambda_{\ell(\lambda)}}$, so the lexicographically least term of $s_\lambda$ is $h_{\lambda_1}h_{\lambda_2\cdots\lambda_{\ell(\lambda)}}=h_\lambda$, as desired.\\
\item By the first part, the transition matrix from the Schur function basis to the complete homogeneous symmetric function basis, when the indices are ordered using reverse lexicographic order, is lower triangular with $1$'s along the main diagonal. Therefore, so is the inverse matrix. \end{enumerate}\end{proof}

\begin{corollary} \label{cor:observation} Let $f$ be a nonzero symmetric function. Expand $f$ in the complete homogeneous symmetric function basis and separate the lexicographically least term to write $$f=c_\lambda h_\lambda+\sum_{\mu>_{lex}\lambda} c_\mu h_\mu,$$ where $c_\lambda\neq 0$. Now if $f$ is Schur-positive, then $c_\lambda>0$. \end{corollary}

\begin{proof} By Part 2 of Lemma \ref{lem:observation}, a term of the Schur function expansion of $f$ is now $c_\lambda s_\lambda$.\end{proof}

We make use of the following convenient reformulation of the Jacobi-Trudi identity for ribbon Schur functions, after a definition.

\begin{definition} \label{def:mulcoarse} \cite[Equation (2.4)]{btvw} Let $\alpha$ be a composition. Recall that a composition $\beta$ is a \emph{coarsening} of $\alpha$ if it can be obtained by adding together adjacent parts of $\alpha$. We define $\mathcal{M}(\alpha)$ to be the multiset of partitions determined by all coarsenings of $\alpha$ and we call it the \emph{multiset of coarsenings of $\alpha$}. For a partition $\lambda$ we denote by $m_\alpha(\lambda)$ its multiplicity in $\mathcal{M}(\alpha)$. \end{definition}

\begin{example} \label{ex:mulcoarse} When $\alpha=1212$ we have that $$\mathcal{M}(\alpha)=\{2211, 321, 321, 321, 33, 42, 51, 6\}\hbox{ and }m_{1212}(321)=3.$$ \end{example}

\begin{theorem} \label{thm:jtcoarse} \cite[Equation (2.6)]{btvw} Let $\alpha$ be a composition. Then $$r_\alpha=(-1)^{\ell(\alpha)}\sum_{\lambda\in\mathcal{M}(\alpha)}(-1)^{\ell(\lambda)}h_\lambda.$$ \end{theorem}

Let us summarize where we are. By Corollary \ref{cor:observation}, we have a necessary condition for a  difference $\ra-\rb$ to be Schur-positive, based on the lexicographically least term appearing in its complete homogeneous symmetric function expansion. By Theorem \ref{thm:jtcoarse}, the complete homogeneous symmetric function expansion of  $\ra$ can be determined by finding the multiplicities of partitions in $\mathcal{M}(\alpha)$. Therefore, our next task is to count the multiplicities of the lexicographically least partitions in $\mathcal{M}(\alpha)$ and $\mathcal{M}(\beta)$. For $\alpha\in \mathcal{R}((a+1)^na^m)$, the lexicographically least partitions that can appear in $\mathcal{M}(\alpha)$ are of the form $$\lambda_k=(2a)^k(a+1)^na^{m-2k}.$$ We now find that the multiplicities of these partitions are intimately connected to the quasi-profile $q(\alpha)$ of $\alpha$. This will be exactly what we need to prove Theorem \ref{thm:smalls}.

\begin{theorem} \label{thm:ftom} Let $\alpha\in \mathcal{R}((a+1)^na^m)$ and $k\geq 0$. Let $\lambda_k$ be the partition $(2a)^k(a+1)^na^{m-2k}$. There is an integer $C_k$ depending only on $m$, $n$, and $q_j(\alpha)$ for $0\leq j\leq k-2$ such that $$m_\alpha(\lambda_k)=C_k-(-1)^kq_{k-1}(\alpha).$$ \end{theorem}

Before we embark on the proof of Theorem \ref{thm:ftom}, we will make some remarks, work through examples, and develop some tools.

\begin{remark} \label{rem:ck} We emphasize that the integer $C_k$ does not depend on the profile $p(\alpha)$. In other words, to count the multiplicity $m_\alpha(\lambda_k)$ one does not require the profile $p(\alpha)$, which is equivalent to knowing $\alpha$, but only the quasi-profile $q(\alpha)$. In addition, one specifically requires the values $q_j(\alpha)$ for $0\leq j\leq k-1$, and the dependence on $q_{k-1}(\alpha)$ is known. \end{remark}

\begin{example} \label{ex:adjacent1} When $k=1$, the partition $\lambda_1=(2a)(a+1)^na^{m-2}$ arises from joining a pair of adjacent $a$'s in $\alpha$. In each string of $p_i\geq 1$ consecutive $a$'s separated by $(a+1)$'s, we have $(p_i-1)$ such pairs. However, when $p_i=0$, we have $0$ rather than $-1$ pairs, so for each such $i$ we must add one to compensate. The multiplicity is therefore $$m_\alpha(\lambda_1)=\sum_{i:\hbox{ }p_i\geq 1}(p_i-1)=\sum_{i:\hbox{ }p_i\geq 0}(p_i-1)+|\{i:\hbox{ }p_i=0\}|=(m-(n+1))+q_0(\alpha).$$We have a term $C_1=m-(n+1)$ depending only on the initial parameters, and then the desired dependence on $q_0$.\end{example}

\begin{example} \label{ex:R5346} Consider the elements of $\mathcal{R}(5^34^6)$ $$\alpha_1=44\hbox{ }5\hbox{ }44\hbox{ }5\hbox{ }44\hbox{ }5,\hbox{ }\alpha_2=444\hbox{ }5\hbox{ }44\hbox{ }5\hbox{ }4\hbox{ }5,\hbox{ and }\alpha_3=4444\hbox{ }5\hbox{ }4\hbox{ }5\hbox{ }4\hbox{ }5.$$ We have $$p(\alpha_1)=2220,\hbox{ }p(\alpha_2)=3210,\hbox{ }p(\alpha_3)=4110,\hbox{ }q(\alpha_1)=103,\hbox{ }q(\alpha_2)=1111,\hbox{  }q(\alpha_3)=1201.$$ Let us count the multiplicities of $\lambda_2=8855544$, which arise by joining two pairs of $4$'s. We can check that $$m_{\alpha_1}(\lambda_2)=3,\hbox{ }m_{\alpha_2}(\lambda_2)=2,\hbox{ and }m_{\alpha_3}(\lambda_2)=1.$$ Having fixed $n=3$, $m=6$, and $q_0(\alpha_i)=1$, the multiplicity can be calculated as $$m_{\alpha_i}(\lambda_2)=3-q_1(\alpha_i),$$ a function of only $q_1(\alpha_i)$. \end{example}

We now prepare to prove Theorem \ref{thm:ftom}. 

\begin{notation}\label{not:bin} Let $x$ be an integer and $k$ be a nonnegative integer. Set $$\binom{x}{k}=\begin{cases} 0 & \hbox{if }x<0 \\\frac{x(x-1)\cdots(x-k+1)}{k!} & \hbox{if }x\geq 0\end{cases}\hbox{ and }\binom{x}{k}_g=\frac{x(x-1)\cdots(x-k+1)}{k!},$$ even if $x<0$. \end{notation}

\begin{lemma} \label{lem:choosek} The number of ways to choose $k$ disjoint pairs of $a$'s from a string of $x$ consecutive $a$'s is $\binom{x-k}{k}$. \end{lemma}

\begin{proof} We use induction on $k$. If $k=0$ we indeed have $\binom{x-0}{0}=1$ way. Otherwise, if the leftmost pair uses the $(j-1)$-th and the $j$-th $a$'s from the left, then by induction there are $\binom{(x-j)-(k-1)}{k-1}$ choices for the remaining pairs. Summing over $j$, we have 
$$\sum_{j\geq 2}\binom{(x-j)-(k-1)}{k-1}=\sum_{j\geq 2}(\binom{x-j-k+2}{k}-\binom{x-j-k+1}{k})=\binom{x-k}{k},$$
as desired.\end{proof}

We will also need a helpful identity of Jensen. The following formulation, specialized at $z=-1$, is sufficient for our purposes.

\begin{theorem} \label{thm:jensen} \cite{jensen} Let $x$ and $y$ be integers and $v$ be a nonnegative integer. Then
$$\sum_{u=0}^v\binom{x-u}{u}_g\binom{y-(v-u)}{v-u}_g=\sum_{u=0}^v\binom{x+y-v-u}{v-u}_g(-1)^u.$$\end{theorem}

Now we are ready to prove Theorem \ref{thm:ftom}.\\

\noindent\emph{Proof of Theorem \ref{thm:ftom}.}
We use induction on $k$. When $k=0$ we simply have $m_\alpha(\lambda_0)=1$, so assume the result for $k'<k$. The partition $\lambda_k$ arises from coarsening by adding $k$ disjoint pairs of adjacent $a$'s from among the $(n+1)$ strings of $a$'s in $\alpha$, whose lengths are given by $p(\alpha)$. Note that even if $a=1$, this is the only way to form $\lambda_k$. We sum over the different ways of choosing how these $k$ pairs are distributed, indexed by nonnegative tuples $\eta=\eta_1\cdots\eta_{n+1}$, and by Lemma \ref{lem:choosek} there are $\binom{p_i-\eta_i}{\eta_i}$ ways to choose $\eta_i$ pairs from a string of $p_i$ $a$'s. We have  $$m_\alpha(\lambda_k)=\sum_{\eta_1+\cdots+\eta_{n+1}=k}\prod_i\binom{p_i-\eta_i}{\eta_i}.$$ From here there are two steps. First we replace the binomial coefficient $\binom{p_i-\eta_i}{\eta_i}$ by the polynomial $\binom{p_i-\eta_i}{\eta_i}_g$, using the induction hypothesis to show how the correction is related to the $q_j(\alpha)$ for small values of $j$. Then we use Theorem \ref{thm:jensen} to show that what remains depends only on $m$ and $n$.\\

For our first step, we write \begin{align}\nonumber m_\alpha(\lambda_k)&=\sum_{\eta_1+\cdots+\eta_{n+1}=k}\prod_i\binom{p_i-\eta_i}{\eta_i}\\\nonumber&=\sum_{\eta_1+\cdots+\eta_{n+1}=k}\prod_{i:\hbox{ }p_i\geq\eta_i}\binom{p_i-\eta_i}{\eta_i}_g\\\nonumber&=\sum_{\eta_1+\cdots+\eta_{n+1}=k}\prod_{\substack{i:\hbox{ }p_i\geq\eta_i\\\hbox{or }p_i=0}}\binom{p_i-\eta_i}{\eta_i}_g-D_0(\alpha,k)\\\nonumber&=\sum_{\eta_1+\cdots+\eta_{n+1}=k}\prod_{\substack{i:\hbox{ }p_i\geq\eta_i\\\hbox{or }p_i\leq 1}}\binom{p_i-\eta_i}{\eta_i}_g-D_0(\alpha,k)-D_1(\alpha,k)\\\nonumber&=\cdots=\sum_{\eta_1+\cdots+\eta_{n+1}=k}\prod_{\substack{i:\hbox{ }p_i\geq\eta_i\\\hbox{or }p_i\leq j'}}\binom{p_i-\eta_i}{\eta_i}_g-\sum_{j=0}^{j'}D_j(\alpha,k)\\\nonumber&=\cdots=\sum_{\eta_1+\cdots+\eta_{n+1}=k}\prod_{\substack{i:\hbox{ }p_i\geq\eta_i\\\hbox{or }p_i\leq k-1}}\binom{p_i-\eta_i}{\eta_i}_g-\sum_{j=0}^{k-1}D_j(\alpha,k)\\\nonumber&=\sum_{\eta_1+\cdots+\eta_{n+1}=k}\prod_i\binom{p_i-\eta_i}{\eta_i}_g-\sum_{j=0}^{k-1}D_j(\alpha,k),\end{align} where the $D_j(\alpha,k)$ represent the correction incurred by terms where $p_i=j$ and $\eta_i>j$. We will show that the $D_j(\alpha,k)$ depend only on $m$, $n$, and $q_{j'}(\alpha)$ for $j'\leq k-2$, with the exception of $D_{k-1}=(-1)^kq_{k-1}(\alpha)$. \\

As a short practice calculation to aid with the full calculation, suppose that $p_1=0$. Whenever $\eta_1=r>0$, then in order to compensate for adding a $p_1=0$ term, we will now need to subtract $$\binom{0-r}{r}_g\prod_{i\geq 2:\hbox{ }p_i\geq\eta_i}\binom{p_i-\eta_i}{\eta_i}_g.$$ Letting $p'=p_2\cdots p_{n+1}$ and $\eta'=\eta_2\cdots\eta_{n+1}$ be the tuples with the first parts removed, and $\alpha'$ be such that $p(\alpha')=p'$, we are subtracting in total $$\sum_{r=1}^{k}\binom{0-r}{r}_g\sum_{\eta'_1+\cdots+\eta'_n=k-r}\prod_{i:\hbox{ }p'_i\geq\eta'_i}\binom{p'_i-\eta'_i}{\eta'_i}_g=\sum_{r=1}^k\binom{-r}{r}_g m_{\alpha'}(\lambda_{k-r});$$ by our induction hypothesis, $m_{\alpha'}(\lambda_{k-r})$ depends only on $m$, $n$, and $q_j$ for $j\leq k-r-1\leq k-2$.\\

Now we will calculate $D_0(\alpha,k)$. If we have multiple zeroes in $p(\alpha)$ we sum over the number $t$ of indices $i$ with $p_i=0$ for which $\eta_i>0$. Let $p^t$ be a tuple with $t$ $0$'s removed, $\eta^t$ be the tuple with the corresponding indices removed, and $\alpha^t$ such that $p(\alpha^t)=p^t$. Then we are subtracting
\begin{align}\nonumber D_0(\alpha,k)&=\sum_{t=1}^{q_0(\alpha)}\binom{q_0(\alpha)}{t}\sum_{r_1,\ldots,r_t\geq1}\left(\prod_{t'=1}^t\binom{0-r_{t'}}{r_{t'}}_g\right)\sum_{\substack{\eta^t_1+\cdots+\eta^t_{n+1-t}\\=k-\sum_{t'}r_{t'}}}\prod_{i:\hbox{ }p^t_i\geq\eta^t_i}\binom{p^t_i-\eta^t_i}{\eta^t_i}_g\\\nonumber&=\sum_{t=1}^{q_0(\alpha)}\binom{q_0(\alpha)}{t}\sum_{r_1,\ldots,r_t\geq 1}\left(\prod_{t'=1}^t\binom{-r_{t'}}{r_{t'}}_g\right)m_{\alpha^t}(\lambda_{k-\sum_{t'}r_{t'}})\end{align}
where the first sum counts the number of indices $i$ with $\eta_i>0$ and the following part represents the contribution from these $t$ indices. The $r_{t'}$ are the values of $\eta_i$ at these indices, which are at least 1. The last sum is the multiplicity $m_{\alpha^t}(\lambda_{k-\sum_{t'}r_{t'}})$,
which again by our induction hypothesis only depends on $m$, $n$, and $q_j(\alpha)$ with $j\leq k-\sum_{t'}r_{t'}-1\leq k-2$, and so $D_0(\alpha,k)$ does too. \\

So we now have \begin{equation} \label{eq:d0}m_\alpha(\lambda_k)=\sum_{\eta_1+\cdots+\eta_{n+1}=k}\prod_{\substack{i:\hbox{ }p_i\geq\eta_i\\\hbox{ or }p_i=0}}\binom{p_i-\eta_i}{\eta_i}_g-D_0(\alpha,k).\end{equation}

By the same argument, the contribution arising from terms with $p_i=1$ is $$D_1(\alpha,k)=\sum_{t=1}^{q_1(\alpha)}\binom{q_1(\alpha)}{t}\sum_{r_1,\ldots,r_t\geq 2}\left(\prod_{t'=1}^t\binom{1-r_{t'}}{r_{t'}}_g\right)\sum_{\substack{\eta^t_1+\cdots+\eta^t_{n+1-t}\\=k-\sum_{t'}r_{t'}}}\prod_{\substack{i:\hbox{ }p^t_i\geq\eta^t_i\\\hbox{ or }p^t_i=0}}\binom{p^t_i-\eta^t_i}{\eta^t_i}_g,$$ where $p^t$ is the tuple with $t$ $1$'s removed and $\eta^t$ has the corresponding indices removed. Letting $\alpha^t$ be such that $p(\alpha^t)=p^t$, now this last sum by rearranging Equation \eqref{eq:d0} is $$\sum_{\substack{\eta^t_1+\cdots+\eta^t_{n+1-t}\\=k-\sum_{t'}r_{t'}}}\prod_{\substack{i:\hbox{ }p^t_i\geq\eta^t_i\\\hbox{ or }p^t_i=0}}\binom{p^t_i-\eta^t_i}{\eta^t_i}_g=m_{\alpha^t}(\lambda_{k-\sum_{t'}r_{t'}})+D_0(\alpha^t,k-\sum_{t'=1}^tr_{t'}).$$ By our induction hypothesis this depends only on $m$, $n$, and $q_j(\alpha)$ for $j\leq k-2$, and so $D_1(\alpha,k)$ does too. \\

In general, the contribution arising from terms with $p_i=j$ is $$D_j(\alpha,k)=\sum_{t=1}^{q_j(\alpha)}\binom{q_j(\alpha)}{t}\sum_{r_1,\ldots,r_t\geq j+1}\left(\prod_{t'=1}^t\binom{j-r_{t'}}{r_{t'}}_g\right)\sum_{\substack{\eta^t_1+\cdots+\eta^t_{n+1-t}\\=k-\sum_{t'}r_{t'}}}\prod_{\substack{i:\hbox{ }p^t_i\geq\eta^t_i\\\hbox{ or } p^t_i\leq j-1}}\binom{p^t_i-\eta^t_i}{\eta^t_i}_g,$$ where $p^t$ is the tuple with $t$ $j$'s removed, $\eta^t$ has the corresponding indices removed, and letting $\alpha^t$ be such that $p(\alpha^t)=p^t$, the latter sum is $$\sum_{\substack{\eta^t_1+\cdots+\eta^t_{n+1-t}\\=k-\sum_{t'}r_{t'}}}\prod_{\substack{i:\hbox{ }p^t_i\geq\eta^t_i\\\hbox{ or } p^t_i\leq j-1}}\binom{p^t_i-\eta^t_i}{\eta^t_i}_g=m_{\alpha^t}(\lambda_{k-\sum_{t'}r_{t'}})+\sum_{j'=0}^{j-1}D_{j'}(\alpha^t,k-\sum_{t'=1}^tr_{t'}),$$ so by our induction hypothesis if $j\leq k-2$, then $D_j(\alpha,k)$ depends only on $m$, $n$, and $q_{j'}(\alpha)$ for $j'\leq k-2$. Furthermore, our last term $D_{k-1}(\alpha,k)$ is 

$$\sum_{t=1}^{q_{k-1}(\alpha)}\binom{q_{k-1}(\alpha)}{t}\sum_{r_1,\ldots,r_t\geq k}\left(\prod_{t'=1}^t\binom{k-1-r_{t'}}{r_{t'}}_g\right)\sum_{\substack{\eta^t_1+\cdots+\eta^t_{n+1-t}\\=k-\sum_{t'}r_{t'}}}\prod_{\substack{i:\hbox{ }p^t_i\geq\eta^t_i\\\hbox{ or }p_i\leq k-2}}\binom{p^t_i-\eta^t_i}{\eta^t_i}_g.$$ 

However, here the only way to have $r_1,\ldots,r_t\geq k$ and a nonnegative tuple $\eta^t$ summing to $k-\sum_{t'} r_{t'}$ is when $t=1$, $r_1=k$, and $\eta^1$ is all zeroes, so we have $$D_{k-1}(\alpha,k)=\binom{q_{k-1}(\alpha)}{1}\binom{k-1-k}{k}_g=\binom{-1}{k}_gq_{k-1}(\alpha)=(-1)^kq_{k-1}(\alpha).$$ So our original count is now $$m_\alpha(\lambda_k)=\sum_{\eta_1+\cdots+\eta_{n+1}=k}\prod_i\binom{p_i-\eta_i}{\eta_i}_g-\sum_{j=0}^{k-2}D_j(\alpha,k)-(-1)^kq_{k-1}(\alpha),$$ where the $D_j(\alpha,k)$ are by our induction hypothesis functions of only $m$, $n$, and $q_{j'}(\alpha)$ for $j'\leq k-2$. This completes the first step. \\

Now it remains to show that the sum $$f(p)=\sum_{\eta_1+\cdots+\eta_{n+1}=k}\prod_i\binom{p_i-\eta_i}{\eta_i}_g$$ depends only on $m$ and $n$. Moreover, it suffices to show that $f(p)=f(p')$, where $$p'=(p_1+1)(p_2-1)p_3\cdots p_{n+1}.$$ This is because $f(p)$ is symmetric in the $p_i$ and any tuple with fixed size $m$ and fixed length $n+1$ can be achieved by a sequence of these steps. We separate the first two parts of $\eta$, summing over different values of $u=\eta_1+\eta_2$, and apply Theorem \ref{thm:jensen} to conclude that
 \begin{align}\nonumber f(p)&=\sum_{\eta_1+\cdots+\eta_{n+1}=k}\prod_i\binom{p_i-\eta_i}{\eta_i}_g\\\nonumber&=\sum_{u=0}^k\sum_{\eta_3+\cdots+\eta_{n+1}=k-u}\left(\sum_{\eta_1=0}^u\binom{p_1-\eta_1}{\eta_1}_g\binom{p_2-(u-\eta_1)}{u-\eta_1}_g\right)\prod_{i\geq 3}\binom{p_i-\eta_i}{\eta_i}_g\\\nonumber&=\sum_{u=0}^k\sum_{\eta_3+\cdots+\eta_{n+1}=k-u}\left(\sum_{\eta_1=0}^u\binom{p_1+p_2-u-\eta_1}{u-\eta_1}_g(-1)^u\right)\prod_{i\geq 3}\binom{p_i-\eta_i}{\eta_i}_g\\\nonumber&=\sum_{u=0}^k\sum_{\eta_3+\cdots+\eta_{n+1}=k-u}\left(\sum_{u=0}^v\binom{p_1+1-\eta_1}{\eta_1}_g\binom{p_2-1-(u-\eta_1)}{u-\eta_1}_g\right)\prod_{i\geq 3}\binom{p_i-\eta_i}{\eta_i}_g\\\nonumber&=\sum_{\eta_1+\cdots+\eta_{n+1}=k}\prod_i\binom{p'_i-\eta_i}{\eta_i}_g=f(p').\end{align}
 This concludes the proof. \hfill\qed

\begin{remark} \label{rem:ck2} From the induction in the proof of Theorem \ref{thm:ftom}, we can show that $C_k$ is an integer-valued polynomial in $m$, $n$, and $q_{j'}(\alpha)$ for $j'\leq k-1$ with rational coefficients. Moreover, the exponents of any monomial $\prod_{j'} q_{j'}^{e_{j'}}$ appearing here satisfy $\sum_{j'} (j'+1)e_{j'}\leq k$. \end{remark}

Now that we have proven Theorem \ref{thm:ftom}, it is a straightforward consequence to prove the goal of this section, namely Theorem \ref{thm:smalls}, which states that for $\alpha,\beta\in \mathcal{R}((a+1)^na^m)$, $$\hbox{ if }\ra\geq_s\rb,\hbox{ then }q(\alpha)\leq_{lex}q(\beta).$$
\emph{Proof of Theorem \ref{thm:smalls}.} Suppose towards a contrapositive that $q(\alpha)>_{lex}q(\beta)$ and that $q(\alpha)$ exceeds $q(\beta)$ by $z$ at smallest index $k$. By Theorem \ref{thm:ftom}, the multiplicities $m_\alpha(\lambda_{k'})=m_\beta(\lambda_{k'})$ for $k'\leq k$, and $$m_\alpha(\lambda_{k+1})-m_\beta(\lambda_{k+1})=-(-1)^{k+1}(q_k(\alpha)-q_k(\beta))=-(-1)^{k+1}z.$$ Therefore, by Theorem \ref{thm:jtcoarse} the difference $\ra-\rb$ has as the lexicographically least term \begin{align*}(-1)^{\ell(\alpha)}(-1)^{\ell(\lambda_{k+1})}(m_\alpha(\lambda_{k+1})-m_\beta(\lambda_{k+1}))h_{\lambda_{k+1}}=&-(-1)^{\ell(\alpha)}(-1)^{\ell(\alpha)-(k+1)}(-1)^{k+1}zh_{\lambda_{k+1}}\\=&-zh_{\lambda_{k+1}},\end{align*} so by Corollary \ref{cor:observation} can not be Schur-positive, so $\ra\ngeq_s\rb$.\hfill\qed

\begin{remark} \label{rem:se} In the proof of Theorem \ref{thm:shortends} we determined the lexicographically largest term that appears in a difference. Now, in the proof of Theorem \ref{thm:smalls}, we determined the lexicographically smallest term. \end{remark}

\begin{corollary} \label{cor:even} Suppose that $m=d(n+1)$ for some $d$. Then $$\alpha=a^d(a+1)a^d\cdots a^d(a+1)a^d$$ is a maximal element of $\mathcal{R}((a+1)^na^m)$. \end{corollary}
\begin{proof} The ribbon $\alpha$ is the unique element of $\mathcal{R}((a+1)^na^m)$ that minimizes $q(\alpha)$.\end{proof}

\begin{remark} \label{rem:smallmaxforposets} According to \eqref{eq:smallscor} of Corollary \ref{cor:maxforposets}, Corollary \ref{cor:even} confirms that the box diagonal diagram for $\mathcal{R}((a+1)^na^{d(n+1)})$ is indeed a maximal element, corroborating Conjecture~\ref{conj:maxel} in this case. \end{remark}

\begin{example} \label{ex:maxR} The ribbon $$\alpha=444\hbox{ }5\hbox{ }444\hbox{ }5\hbox{ }444\hbox{ }5\hbox{ }444\hbox{ }5\hbox{ }444$$ is a maximal element of $\mathcal{R}(5^44^{15})$. \end{example}

\begin{corollary} \label{cor:incompar} 
\hspace{2pt}
\begin{enumerate}
\item For any $t\geq 0$ the ribbons $$\alpha=a(a+1)(a+1)aaaa^t\hbox{ and }\beta=(a+1)aa(a+1)aaa^t$$ are incomparable.\\
\item For any fixed ribbon $\gamma$ with only parts $a$ and $(a+1)$, the ribbons $$\alpha=aa(a+1)(a+1)(a+1)\cdot\gamma\hbox{ and }\beta=(a+1)a(a+1)a(a+1)\cdot\gamma$$ are incomparable.
\end{enumerate} It follows immediately that the chains identified in Theorem \ref{thm:chains} are the \emph{only} cases when $\mathcal{R}((a+1)^na^m)$ is a chain. \end{corollary}
\begin{proof} \begin{enumerate}
\item By Theorem \ref{thm:shortends}, we know that $\rb\ngeq_s\ra$. Conversely, we calculate $$p(\alpha)=10(t+3),\hbox{ }q(\alpha)=1100\cdots01,\hbox{ }p(\beta)=02(t+2),\hbox{ and }q(\beta)=1010\cdots01$$ (or $q(\beta)=102$ if $t=0$). So $q(\alpha)>_{lex}q(\beta)$, and by Theorem \ref{thm:smalls}, we know that $\ra\ngeq_s\rb$.\\
\item By Theorem \ref{thm:shortends}, we know that $\rb\ngeq_s\ra$. Conversely, we calculate $$p(\alpha)=200\cdot p(\gamma),\hbox{ }q(\alpha)=201+q(\gamma),\hbox{ }p(\beta)=011\cdot p(\gamma),\hbox{ and }q(\beta)=120+q(\gamma),$$ where $\cdot$ is concatenation as before, and $+$ is pointwise addition of tuples. So $q(\alpha)>_{lex}q(\beta)$, and by Theorem \ref{thm:smalls}, we know that $\ra\ngeq_s\rb$.\end{enumerate}\end{proof}

\subsection{Minimal equitable ribbons} \label{subsec:minimal}

Having developed this machinery, we may now answer Conjecture~\ref{conj:minel} by proving that indeed the conjectured minimal element of $\mathcal{R}((a+1)^na^m)$ is minimal. As independent motivation, recall that Theorem \ref{thm:shortends} tells us that a large element of $\mathcal{R}((a+1)^na^m)$ tends to have short rows near the ends. Also, Theorem \ref{thm:smalls} tells us that a large element of $\mathcal{R}((a+1)^na^m)$ tends to have its long rows evenly spaced out. In view of these, it could be expected that the minimal element should be the one that is farthest from achieving either of these, and it is with this result that we conclude.

\begin{theorem} \label{thm:minimal} The ribbon $$\alpha=(a+1)^{\lceil\frac{n}{2}\rceil}a^m(a+1)^{\lfloor\frac{n}{2}\rfloor}$$ is a minimal element of $\mathcal{R}((a+1)^na^m)$. \end{theorem}

\begin{proof} If $n=0$ or $m=0$ then $\alpha$ is the only element of $\mathcal{R}((a+1)^na^m)$ so the statement is trivially true, and if $n=1$ or $m=1$ then $\mathcal{R}((a+1)^na^m)$ is a chain so the result follows from Theorem \ref{thm:chains}. So suppose that $m,n\geq 2$. Let $\beta\in \mathcal{R}((a+1)^na^m)$ be a ribbon with $\beta<_s\alpha$. By Theorem \ref{thm:smalls}, $q(\beta)\geq_{lex}q(\alpha)$. However, $q(\alpha)=n0^{m-2}1$ achieves the lexicographically maximum possible value of $q(\alpha)$, so $q(\beta)=q(\alpha)=n0^{m-2}1$ and in $\beta$ all of the $a$'s are together, so we have $$\beta=(a+1)^{n-t}a^m(a+1)^t$$ for some $0\leq t<\lfloor\frac{n}{2}\rfloor$. By Theorem \ref{thm:shortends}, $t\geq 1$.\\

Because $q(\alpha)=q(\beta)$, we have by Theorem \ref{thm:ftom} that the multiplicities in $\mathcal{M}(\alpha)$ and $\mathcal{M}(\beta)$ agree at every $\lambda_k=(2a)^k(a+1)^na^{m-2k}$. So we now move to count the multiplicities at the next smallest partitions. \\

Let us consider the partitions $$\mu_k=(2a+2)^k(a+1)^{n-2k}a^m,$$ which arise from joining $k$ pairs of $(a+1)$'s. Note that simply by switching the roles of the $a$'s and the $(a+1)$'s, we have $$m_{(a+1)^{n-t}a^m(a+1)^t}((2a+2)^k(a+1)^{n-2k}a^m)=m_{a^{n-t}(a+1)^ma^t}((2a)^k(a+1)^ma^{n-2k}),$$ because the right hand side now counts the number of ways to join $k$ pairs of $a$'s. Therefore, the problem of counting the multiplicities of the $\mu_k$ reduces to the problem of counting those of the $\lambda_k$ from Theorem \ref{thm:ftom}. It follows from Theorem \ref{thm:ftom} that $$m_\alpha(\mu_k)=m_\beta(\mu_k)\hbox{ for }k\leq t\hbox{ and }m_\alpha(\mu_{t+1})-m_\beta(\mu_{t+1})=(-1)^{t+1}.$$

We should note that the multiplicities in $\mathcal{M}(\alpha)$ and $\mathcal{M}(\beta)$ not only agree at every $\mu_k$ for $k\leq t$, but they also agree at any partition that is obtained from further coarsening by adding any $a$'s together in any way. This is  because both $\alpha$ and $\beta$ have all $m$ $a$'s together, and they are not affected by the coarsenings that form the $(2a+2)$'s. So we need not worry about any of these partitions, even though when $a=1$ they may be smaller lexicographically than some of the $\mu_k$.\\

Similarly, consider the partitions
 $$\mu'_k=(2a+2)^k(2a+1)(a+1)^{n-2k-1}a^{m-1},$$ which arise from joining either the leftmost or the rightmost $a$ to a neighbouring $(a+1)$, then joining $k$ pairs of $(a+1)$'s as before when we counted the multiplicity of $\mu_k$. Then by switching $a$'s and $(a+1)$'s again, the problem reduces to that of counting multiplicties of $\lambda_k$ as in Theorem \ref{thm:ftom}. We have \begin{align}\nonumber m_{(a+1)^{n-t}a^m(a+1)^t}((2a+2)^k(2a+1)(a+1)\cdots)&=m_{(a+1)^{n-t-1}a^{m-1}(a+1)^t}((2a+2)^k(a+1)\cdots)\\\nonumber&+m_{(a+1)^{n-t}a^{m-1}(a+1)^{t-1}}((2a+2)^k(a+1)\cdots)\\\nonumber&=m_{a^{n-t-1}(a+1)^{m-1}a^t}((2a)^k(a+1)\cdots)\\\nonumber&+m_{a^{n-t}(a+1)^{m-1}a^{t-1}}((2a)^k(a+1)\cdots)\end{align} and as before, using Theorem \ref{thm:ftom} we find that $$m_\alpha(\mu'_k)=m_\beta(\mu'_k)\hbox{ for }k\leq t-1\hbox{ and }m_\alpha(\mu'_t)-m_\beta(\mu'_t)=(-1)^t.$$ Again, because in $\alpha$ and $\beta$ all the $a$'s are together, we need not consider partitions that arise from $\mu'_k$ for $k\leq t-1$ by further coarsening by adding $a$'s in any way.\\

Finally, we apply the same process to the partitions $$\mu''_k=(2a+2)^k(2a+1)^2(a+1)^{n-2k-2}a^{m-2},$$ which arise from now joining both the leftmost and the rightmost $a$ (recall that $m\geq 2$) to neighbouring $(a+1)$'s, then joining $k$ pairs of $(a+1)$'s as before. Using a similar reduction and applying Theorem \ref{thm:ftom}, we find that $$m_{(a+1)^{n-t}a^m(a+1)^t}((2a+2)^k(2a+1)^2(a+1)\cdots)=m_{a^{n-t}(a+1)^ma^t}((2a+1)^2(2a)^k(a+1)\cdots)$$ and  $$m_\alpha(\mu''_k)=m_\beta(\mu''_k)\hbox{ for }k\leq t-1\hbox{ and }m_\alpha(\mu''_t)-m_\beta(\mu''_t)=(-1)^t.$$ Again, because in $\alpha$ and $\beta$ all the $a$'s are together, we need not consider partitions that arise from $\mu''_k$ for $k\leq t-1$ by further coarsening by adding $a$'s in any way.\\

In conclusion, the lexicographically smallest partition at which $\mathcal{M}(\alpha)$ and $\mathcal{M}(\beta)$ differ is $\mu'_t$, at which $$m_\alpha(\mu'_t)-m_\beta(\mu'_t)=(-1)^t.$$ Therefore, by Theorem \ref{thm:jtcoarse} the difference $\ra-\rb$ has as the lexicographically least term $$(-1)^{\ell(\alpha)}(-1)^{\ell(\mu'_t)}(m_\alpha(\mu'_t)-m_\beta(\mu'_t))h_{\mu'_t}=(-1)^{\ell(\alpha)}(-1)^{\ell(\alpha)-t-1}(-1)^th_{\mu'_t}=-h_{\mu'_t},$$ so by Corollary \ref{cor:observation} can not be Schur-positive, a contradiction.\end{proof}

\section*{Acknowledgements}\label{sec:acknow} The authors would like to thank Peter McNamara for very helpful conversations, for generously sharing ideas, and for suggesting fruitful avenues of research. They would also like to thank Fran\c{c}ois Bergeron for sharing the beautiful calculation for the rarity of Schur-positivity with them, and the referee for thoughtful comments.



\end{document}